\documentclass[12pt]{article}
\usepackage[utf8]{inputenc}

\usepackage[left=1cm, top=1cm, right=1cm, bottom=1cm, includeheadfoot, headheight=0pt, a4paper]{geometry}
\usepackage{fancyhdr}
\usepackage{lastpage}
\usepackage{amsmath}
\usepackage{amssymb}
\usepackage{graphicx}
\usepackage{enumitem}
\usepackage{multicol}
\usepackage{amssymb}
\usepackage{amsmath}
\usepackage{amsfonts}
\usepackage[utf8]{inputenc}
\usepackage{amsthm}
\usepackage{tikz-cd}
\usepackage{geometry}

\newtheorem{thm}{Theorem}[section]
\newtheorem{lemma}[thm]{Lemma}

\newtheorem{conj}[thm]{Conjecture}
\newtheorem{prob}[thm]{Problem}
\theoremstyle{remark}
\newtheorem{remark}[thm]{Remark}
\newcommand{\p}{\mathfrak{p}}
\newcommand{\OO}{\mathcal{O}}
\renewcommand{\thefootnote}{\fnsymbol{footnote}}

\newcommand\blfootnote[1]{%
  \begingroup
  \renewcommand\thefootnote{}\footnote{#1}%
  \addtocounter{footnote}{-1}%
  \endgroup
}

\title{On the fourth power level of $\mathfrak{p}$-adic completions  of biquadratic number fields }

\author{Kazimierz Chomicz}
\date{March 2025}
\begin{document}
%\keywords{Fourth level, Hensels Lemma, biquadratic number fields.}
%\subjclass[2020]{11R16, 11E76}

%%%%%%%%%%%%%%%%%%%%%%%%%%%%%%%%%%%
%%%%%%%%%%%%%%%%%%%%%%%%%%%%%%%%%%%

\maketitle

\abstract{Let $K$ be a number field and $\mathfrak{p} \mid (2)$ be a prime ideal. We compute the fourth level of the $\mathfrak{p}$-adic completions of $K$ when the ramification index is $4$ and the inertial degree is trivial for the ideal $\mathfrak{p}$. This enables the computation of the fourth level of any $\mathfrak{p}$-adic completion of any quartic number field. Here we apply this result to biquadratic number fields and obtain lower bounds for the fourth level of such number fields.}

\section{Introduction}

\blfootnote{\text{Keywords: Fourth level, Hensels Lemma, biquadratic number fields.}}
\blfootnote{\text{2020 Mathematics Subject Classification: 11R16, 11E76, 11P05}}

Let $K$ be a field. We define the fourth level of $K$, denoted by $s_4(K)$ as the smallest positive integer $g$ such that $x_1^4 + x_2^4 + \dots + x_g^4 = -1$ is solvable in $K$. Of course, such $g$ may not exist, in this case we put $s_4(K)=\infty $. Usually, finding lower or upper bounds on $s_4$ is a hard and technical task. This problem seems to be of similar difficult as that of computing the Waring numbers (Pythagoras numbers) of a field (cf. \cite{km2024}). In general, the Waring number $w_d(K)$ is the smallest positive integer $g$ such that any sum of $d$th powers can be expressed as a sum of at most $g$ $d$th powers. Of course such number may not exist. If $K$ is a field with a finite second level $s_2$ then all of its higher levels are finite also (cf. \cite{joly1970}). 

There is a close connection between the levels of a field and its Waring numbers. Assume that characteristic of $K$ is 0, then
$$w_n(K)\leq  nG(n)(s_n(K)+1). $$
Here, $G(n)$ is the big $G$ function in the classical Waring problem (see \cite{km2024} for some of the recent developments on the Waring problem in various cases).
Apart from fields with a finite level one can consider rings and fields with infinite second (and higher) level. This is a rich and difficult theory of sums of $2d$th powers in real algebraic geometry as well as in totally real number fields (cf. \cite{ bk2023, bk2024, grimm2015, km2024, kv2023, krasensky2022, krs2022,  ky2023}).

Here, we will be interested in finding lower bounds for $s_4(K)$, where $K$ is a quartic number field, by computing the fourth level of their $\p$-adic completions. Our goal is to generalize results from \cite{paper-per-piotr}, where similar problem was considered for quadratic and cubic number fields. In section \ref{cef} we will build towards results for an arbitrary quartic number field, however later in section \ref{step2}, we will switch to only consider biquadratic number fields $K = \mathbb{Q}(\sqrt{m},\sqrt{n})$ for some integers $n,m$.

For an algebraic number field $K$ let $K_{\mathfrak{p}}$ denote the $\mathfrak{p}$-adic completion of $K$ at a prime ideal $\mathfrak{p}$ of $\mathcal{O}_K$ -- the ring of integers of $K$. We refer the reader for more information about completion to \cite[Chapter 7]{eisenbud1995}. From $K \subseteq K_{\mathfrak{p}}$, we get a lower bound $s_4(K) \geq s_4(K_{\mathfrak{p}})$.  

%It is important to note that this method is equivalent to another, more intuitive method. It is true that $s_4(K) = s_4(\OO_K)$, hence one can compute a lower bound for $s_4(\OO_K)$ by computing $s_4(\OO_K / I)$ for an ideal $I$ of $\OO_K$. As it will turn out in a moment, by Hensels Lemma computation of $s_4(K_{\p})$ is equivalent to computing $s_4$ of a certain quotient ring of $\OO_K$. The best bound which is achievable from the method of completions, will be the same best possible bound which is achievable from the method of computing $s_4(\OO_K / I)$. 

For a prime ideal $\p \nmid (2)$ in $\OO_K$, the computation of the fourth level of $K_\p$ is reduced to the computation of the fourth level of the residue field. The residue field will be a finite field of characteristic different of 2 and those values are known (cf. \cite{paper-per-piotr}). Hence, throughout the paper we will only consider prime divisors of the ideal $(2)$.

From \cite[Theorem 21]{marcus2018} if $(2) = \mathfrak{p}_1^{e_1}\mathfrak{p}_2^{e_2} \cdots \mathfrak{p}_r^{e_r}$ and $K$ is a quartic number field, we have $4 = [K : \mathbb{Q}] = \sum_{i=1}^r e_if_i$, where $e_i = e(\p_i)$ is the ramification index and $f_i = f(\p_i)$ is the inertial degree. In particular this means $ef \leq 4$, where $e=e(\mathfrak{p})$ and $f=f(\mathfrak{p})$ for a fixed prime ideal $\mathfrak{p}$ which divides $(2)$ in $\OO_K$.

Let us now recall four facts directly proved in \cite{paper-per-piotr} (here $s=s_4(K_\p)$).
\begin{thm}\label{1.1} 
We have an equivalence  $e=f=1 \iff s=15$. 
\end{thm}

\begin{thm}\label{1.2} 
 If $f$ is even then $s \leq 2$. If $s=1$ then $4 \mid e$. 

\end{thm}
This theorem implies that $s=2$ for the cases $e=1, f=2$; $e=2, f=2$ and $e=1, f=4$.

\begin{thm}\label{1.3} 
 Let $e=2,\ f=1$. Let $\pi$ be any element of $\mathcal{O}_K $ such that $\mathfrak{p} \parallel \pi$ $($i.e. $\pi \in \mathfrak{p} \setminus \mathfrak{p}^2$ $)$. \\
Then $$s = \begin{cases}
    6 \text{ if } \pi^2 \equiv 2 \pmod{\mathfrak{p}^4}\\
    4 \text{ otherwise.}
\end{cases} $$
\end{thm}

\begin{thm}\label{1.4}
If $e = 3,\ f=1$, then $s=9$, and if $e=1, \ f=3$, then $s=5$.
    
\end{thm}

It is important to note that the above theorems work for any number field $K$, with a prime $\p \mid (2)$ with $e = e(\p)$ and $f = f(\p)$. The reason why these theorems where considered in \cite{paper-per-piotr} is that they where needed for computation of $s$ in the cases of quadratic and cubic number fields.

The only case not covered by the above which occurs for quartic number fields (remembering $ef \leq 4$) is the case $e = 4, \ f=1$. This case will be done by us similarly to Theorem \ref{1.3}, by considering an arbitrary element $\p \parallel \tau $ and producing a similar tree of cases on $s$ (although much more complicated) and will be worked on over any number field with a prime $\p \mid (2)$ with $e=4, \ f=1$. This we will do in the next section, which will provide the first main result of the paper, as this is a generalization of the results from \cite{paper-per-piotr}. Because we will do this over an arbitrary number field, similarly as we are using above theorems, our result can be used also when working with number fields of degree $>4$, whenever the case $e=4, \ f=1$ appears. Section 3 contains an application of our main theorem. In particular, we compute the factorization of $(2)$ for biquadratic number fields and compute the fourth level of the completions. This constitutes of our second result. We finish the paper with remarks which may stimulate further research.

We end this section with some technical details, important for the next section. Notice that because $K$ is a field, its fourth level $s$ is also the smallest number such that $x_1^4 + \dots + x_s^4 + x_{s+1}^4 = 0$ is non trivially solvable in $K$ (we can move $x_{i}^4 \neq 0$ for some $i$ to the other side and divide by it). 

Denote by $v_\p$ the $\p$-adic valuation of $K_\p$.  We say that $K_\mathfrak{p}$ satisfies Hensels Lemmma if for every polynomial $f(x) \in \OO_{K_\p}[x]$ if there exists $ b \in \OO_{K_\p}$ such that $v_\p(f(b))>2v_\p(f'(b))$ then there exists an element $a \in \OO_{K_\p}$ such that $f(a)=0$ and $v_\p(a-b)>v_\p(f'(b))$ (cf. \cite[Theorem 4.1.3]{ep2005}). 

In particular, by the Hensels Lemma for a prime ideal $\p \mid (2)$ we may reduce computing the fourth level of $K_\p$ to the computation of the fourth level of $\OO_K/\p^{4e+1}$. In this case, the quotient is a finite ring. This follows from the fact that from any solution of $x_1^4 + \dots + x_s^4+x_{s+1}^4=0$ in $K_\p$ we can construct one which is in $\OO_{K_\p}$ and has at least one  summand prime with $\p$. To see this we look at a solution $x_1^4 + x_2^4+ \dots + x_s^4 \equiv -1 \pmod{\p^{4e+1}}$ where $\p \nmid x_1$ as $x_1$ being a root of a polynomial $f(x) = x^4 + x_2^4 + \dots + x_s^4 + 1$ in $\OO_K /\p^{4e+1}$. Now we have $v_\p(f(x_1)) \geq 4e + 1 > 4e = 2v_\p(4x_1^3) = 2v_\p(f'(x_1))$, hence by Hensels Lemma we have a solution in $\OO_{K_\p}$. This implies that $s_4(K_\p) = s_4(\OO_K /\p^{4e+1})$ (cf. \cite{par1981}).

\section{Main result} \label{cef}
In this section we will calculate $s = s_4(K_{\mathfrak{p}})$ for any number field $K$ with a prime ideal $\p \mid (2)$ of $\OO_K$ with $e(\p) = e =4,\ f(\p)=f=1$. This is the first result not covered by \cite{paper-per-piotr} and constitutes the first main result of our paper.

Let $\tau$ be an element of $\mathcal{O}_K$ such that $\mathfrak{p} \parallel \tau$ (i.e. $\tau \in \mathfrak{p} \setminus \mathfrak{p}^2 $). We will make a frequent use of following fact: if $f=1$ then for $\mathfrak{p}^{\varphi} \parallel a$ and $\mathfrak{p}^{\varphi} \parallel b$ we have $\mathfrak{p}^{\varphi+1} \mid a+b$. This follows from the fact that $\mathfrak{p}^{\varphi}/\mathfrak{p}^{\varphi+1} \cong \mathcal{O}_K/\mathfrak{p} \cong \mathbb{F}_2$.

By the Hensels Lemma, $s$ is the smallest number such that $x_1^4 + \dots + x_{s+1}^4 \equiv 0 \pmod{\mathfrak{p}^{4e+1}}$ is solvable with $\mathfrak{p} \nmid x_1$. In \cite{paper-per-piotr} it was proven that the exponent $4e+1$ can actually be replaced by $3e+1$. This is useful, because we only need to compute $s_4(\mathcal{O}_K / \mathfrak{p}^{3e+1}) = s_4(\mathcal{O}_K / \mathfrak{p}^{4e+1}) = s_4(K_{\mathfrak{p}})$. Additionally, $\mathcal{O}_K / \mathfrak{p}^{3e+1}$ has strictly fewer elements than $\mathcal{O}_K / \mathfrak{p}^{4e+1}$. Let us start with following lemma
\begin{lemma}
Let $e>1$ and $f=1$. If  $x \equiv y \pmod{\mathfrak{p}^e}$  then $x^4 \equiv y^4 \pmod{\mathfrak{p}^{3e+1}}$.
\end{lemma}
\begin{proof} We have $x^4 - y^4 = (x-y)(x+y)(x^2+y^2)\equiv 0 \pmod{\mathfrak{p}^{3e}}$. This is because $\mathfrak{p}^{e} \mid x-y,x+y,x^2+y^2$. If $x^4-y^4 \not\equiv 0 \pmod{\mathfrak{p}^{3e+1}}$ we must have $\mathfrak{p}^{e} \parallel x-y,x+y,x^2+y^2$. This yields $\mathfrak{p}^{e+1} \mid (x+y)-(x-y)=2y$ or $\mathfrak{p} \mid y$. We now have $\mathfrak{p}^{e+2} \mid(x-y)(x+y) + 2y^2 = x^2 + y^2$, a contradiction. 
\end{proof}

The above lemma allows us to easily verify all the possible cases. For the rest of this section, we put $e=4$ and $f=1$. We can assume that the values of $x_i$ which are not divisible by $\mathfrak{p}$ lie in
\begin{equation}
    \{ 1, \  \tau + 1,\  \tau^2 + 1,\  \tau^3+1, \ \tau+\tau^2+1, \ \tau+\tau^3+1, \ \tau^2+\tau^3+1, \ \tau+\tau^2+\tau^3+1\},
\end{equation}
the values of $x_i$ for which $\mathfrak{p} \parallel x_i$ lie in
\begin{equation}
    \{ \tau, \  \tau + \tau^2, \  \tau+\tau^3, \ \tau+\tau^2+\tau^3\},
\end{equation}
the values of $x_i$ for which $\mathfrak{p}^2 \parallel x_i$ lie in
\begin{equation}
    \{ \tau^2, \  \tau^2+\tau^3\},
\end{equation}
and the values of $x_i$ for which $\mathfrak{p}^3 \parallel x_i$ lie in 
\begin{equation}
    \{\tau^3\}.
\end{equation}
We do not care about values of $x_i$ for which $\mathfrak{p}^4 \mid x_i$, because their fourth powers vanish modulo $\mathfrak{p}^{13}$. %For future referring to above numbers, we say that $a$ is exactly in $\mathfrak{p}^{\varphi}$ if $\mathfrak{p}^{\varphi} \parallel a$ or equivalently $a \in \mathfrak{p}^{\varphi} \setminus \mathfrak{p}^{\varphi+1}$. 

For the remainder of this section, our aim is to find the smallest number of fourth powers of the elements from the above sets, which we need to add up to get 0 modulo $\mathfrak{p}^{13}$. To accomplish that, we will eventually split into cases of positive integers $\alpha$ such that $\mathfrak{p}^{\alpha} \parallel \tau^4+2$. We know that $\mathfrak{p}^{4}$ divides both $\tau^4$ and $2$ and furthermore $\mathfrak{p}^{5}$ does not divide $\tau^4$, nor $2$. We conclude that $\mathfrak{p}^{5} \mid \tau^4 +2$ and hence $\alpha \geq 5$. Eventually, our aim is to consider all cases $\alpha =5,6,\dots, 13$. Some facts however, can be done for more cases at once.

It is important to note that a sum of fourth powers in which an odd number of fourth powers is not divisible by $\mathfrak{p}$, cannot be 0 modulo $\mathfrak{p}^{3e+1}$ (by $f=1$, a sum of two not divisible by $\mathfrak{p}$ elements is divisible by $\mathfrak{p}$, hence modulo $\mathfrak{p}$ a sum with an odd number of such elements is not 0). In order to use Hensel's lemma we only need to consider sums of fourth powers which have at least one fourth power not divisible by $\mathfrak{p}$. Let us define a \emph{good sum} to be a sum of fourth powers, in which there is an even and nonzero number of fourth powers not divisible by $\mathfrak{p}$.
 
\begin{lemma}\label{2.2} If $\mathfrak{p}^{5} \parallel \tau^4+2$ or $\mathfrak{p}^{7} \parallel \tau^4+2$ then $s > 2$. If $\mathfrak{p}^6 \not\parallel \tau^4+2$ then $s > 1$.
\end{lemma}
\begin{proof} We will first prove that $s\neq 2$ in the first two mentioned cases. To prove this, we will consider all possible fourth powers modulo $\mathfrak{p}^{8}$. A fourth power of an element divisible by $\mathfrak{p}^2$, will vanish modulo $\mathfrak{p}^8$. We will only check what are the fourth powers of elements from (1) and (2) modulo $\mathfrak{p}^{8}$:

\begin{center}
\begin{tabular}{@{}l@{}}
    $1^4 \equiv 1$,\\
    $(\tau+1)^4 \equiv  \tau^4 + 2\tau^2 +1$,\\

    $(\tau^2+1)^4 \equiv 1$,\\
 
    $(\tau^3+1)^4 \equiv 1$,\\
       
    $(\tau + \tau^2 + 1)^4 \equiv \tau^4 + 2\tau^2 + 1$,\\
    $(\tau + \tau^3 + 1)^4 \equiv \tau^4 + 2\tau^2 + 1$,\\
    $(\tau^2 + \tau^3 + 1)^4 \equiv 1$,\\
    $(\tau + \tau^2 + \tau^3 + 1)^4 \equiv  \tau^4 + 2\tau^2 + 1$,\\
    
    $\tau^4 \equiv \tau^4$,\\
    $(\tau + \tau^2)^4 \equiv \tau^4$,\\
    $(\tau + \tau^3)^4 \equiv \tau^4$,\\
    $(\tau+\tau^2+\tau^3)^4 \equiv \tau^4$.
\end{tabular}

\end{center}

Hence, all the possible nonzero remainders of fourth powers modulo $\mathfrak{p}^{8}$ are:
\[ \{1, \  \tau^4,\  \tau^4 + 2\tau^2 + 1\}. \]

Having this, we can see that sums of any two not divisible by $\mathfrak{p}$ remainders from the above set (either $1$ or $\tau^4 + 2\tau^2 + 1$) with one reminder divisible by $\mathfrak{p}$ (either 0 or $\tau^4$) modulo $\mathfrak{p}^8$ are 
\[ \{ 2, \ \tau^4+2, \ 2\tau^2 + 2, \ \tau^4 + 2\tau^2+2  \}\]

Now, provided that $\mathfrak{p}^5$ or $\mathfrak{p}^{7} \parallel \tau^4+2$ we can see that the above remainders are nonzero modulo $\mathfrak{p}^8$. This is because $\mathfrak{p}^4 \parallel 2$ and $\mathfrak{p}^6 \parallel 2\tau^2$. Hence $s > 2$ in those two cases. 

As for the case $\mathfrak{p}^6 \not\parallel \tau^4+2$, we see that the sum of two not divisible by $\mathfrak{p}$ fourth powers will be either 2 or $\tau^4 + 2\tau^2+2$ modulo $\mathfrak{p}^8$. Hence, in order to have $s=1$ we have to have $\tau^4+2 + 2\tau^2 \equiv 0 \pmod{\mathfrak{p}^8}$, but because $\mathfrak{p}^6 \parallel 2\tau^2$, we must also have $\mathfrak{p}^6 \parallel \tau^4 +2$, so the conclusion follows. 
\end{proof}

\begin{lemma}\label{2.3}
 If $\mathfrak{p}^{9} \parallel \tau^4+2$ or $\mathfrak{p}^{12} \parallel \tau^4+2$ then $s > 2$.  
 \end{lemma}
\begin{proof} By the previous lemma we already have $s>1$. To determine that $s \neq 2$ in those two cases, we need to consider every good sum of three fourth powers of remainders from (1), (2), (3) and (4). This time computation will be conducted modulo $\mathfrak{p}^{3e+1} = \mathfrak{p}^{13}$. Below we present a list of possible remainders modulo $\mathfrak{p}^{13}$ of fourth powers which will be of interest for us.
\begin{center}
\begin{tabular}{@{}l@{}}
    $1^4 \equiv 1$,\\
    $(\tau+1)^4 \equiv \tau^4 + 4\tau^3 + 6\tau^2 + 4\tau +1$,\\
    $(\tau^2+1)^4 \equiv \tau^8 + 6\tau^4 + 4\tau^2 + 1$,\\
    $(\tau^3+1)^4 \equiv \tau^{12} + 2\tau^6 + 4\tau^3 + 1$,\\
    $(\tau + \tau^2 + 1)^4 \equiv \tau^8 + 2\tau^6 + 3\tau^4 + 2\tau^2 + 4\tau + 1$,\\
    $(\tau + \tau^3 + 1)^4 \equiv \tau^{12} + 2\tau^6 + \tau^4 + 6\tau^2 + 4\tau +1$,\\
    $(\tau^2 + \tau^3 + 1)^4 \equiv \tau^8+2\tau^6+2\tau^4+4\tau^3+4\tau^2+1$,\\
    $(\tau + \tau^2 + \tau^3 + 1)^4 \equiv \tau^{12}+\tau^8+3\tau^4+4\tau^3+2\tau^2+4\tau+1$,\\
    $\tau^4 \equiv \tau^4$,\\
    $(\tau + \tau^2)^4 \equiv \tau^8 + 2\tau^6+ \tau^4$,\\
    $(\tau + \tau^3)^4 \equiv \tau^4$,\\
    $(\tau+\tau^2+\tau^3)^4 \equiv \tau^8+2\tau^6+\tau^4$, \\
    $(\tau^2)^4 \equiv \tau^8$,\\
    $(\tau^2 + \tau^3)^4 \equiv \tau^{12} +\tau^8 $,\\
    $(\tau^3)^4 \equiv \tau^{12}$.
\end{tabular}
\end{center}
Now we need to consider every remainder modulo $\mathfrak{p}^{13}$ which is given as a sum of three remainders from the above (not necessarily distinct) following the rule that two of the remainders are not divisible by $\mathfrak{p}$ and one is. First consider what happens if we add two same remainders not divisible by $\mathfrak{p}$:
\begin{center}
\begin{tabular}{@{}l@{}}
    $2 \cdot (1)^4 \equiv 2$,\\
    $2(\tau+1)^4 \equiv 2\tau^4+4\tau^2+2$,\\
    $2(\tau^2+1)^4 \equiv 2$,\\
    $2(\tau^3+1)^4 \equiv 2$,\\
    $2(\tau+\tau^2+1)^4 \equiv 2\tau^4+4\tau^2+2$,\\
    $2(\tau+\tau^3+1)^4 \equiv 2\tau^4+4\tau^2+2$,\\
    $2(\tau^2+\tau^3+1)^4 \equiv 2$,\\
    $2(\tau+\tau^2+\tau^3+1)^4 \equiv 2\tau^4 + 4\tau^2 + 2 $,
\end{tabular}
\end{center}
so the only remainders modulo $\mathfrak{p}^{13}$ are $\{ 2, \ 2\tau^4+4\tau^2+2 \}$. If to any of these elements, we add a fourth power divisible by $\mathfrak{p}$, to obtain 0 modulo $\mathfrak{p}^{13}$,
this fourth power must be in $\mathfrak{p}^4\setminus \mathfrak{p}^5$. This is because the two elements in the above set are in $\mathfrak{p}^4\setminus \mathfrak{p}^5$. Now, we are adding an element from the set $\{\tau^4,\tau^8+2\tau^6+\tau^4\}$ to an element from $\{ 2,2\tau^4+4\tau^2+2 \}$ and we obtain the following possibilities: 
\[\tau^4+2,\quad (\tau^4+1)(\tau^4+2),\quad  2\tau^4 + 4\tau^2 + (\tau^4+2),\quad  (\tau^4+2\tau^2+1)(\tau^4+2)\] 
modulo $\mathfrak{p}^{13}$. By assumption on $\tau^4+2$, none of them vanish modulo $\mathfrak{p}^{13}$. To make it clear, let us discuss the element $2\tau^4 + 4\tau^2 + (\tau^4+2)$. It is nonzero, since $\mathfrak{p}^8 \parallel 2\tau^4$ and both $4\tau^2$ and $\tau^4+2$ are divisible by $\mathfrak{p}^9$. Hence the element in question is in $\mathfrak{p}^8\setminus \mathfrak{p}^9$ and it does not vanish modulo $\mathfrak{p}^{13}$. This proves that we cannot get 0 modulo $\mathfrak{p}^{13}$ by a good sum of three fourth powers in which the two fourth powers not divisible by $\mathfrak{p}$ are the same. 

As a next step, let us consider sums of two distinct fourth powers not divisible by $\mathfrak{p}$ modulo $\mathfrak{p}^{13}$:
\begin{center}
\begin{tabular}{@{}l@{}}
    $1^4 + (\tau+1)^4 \equiv \tau^4 + 4\tau^3 + 6\tau^2 + 4\tau +2$,\\
    $1^4 + (\tau^2+1)^4 \equiv \tau^8 + 6\tau^4 + 4\tau^2 + 2$,\\
    $1^4 + (\tau^3+1)^4 \equiv \tau^{12} + 2\tau^6 + 4\tau^3 + 2$,\\
    $1^4 + (\tau+\tau^2+1)^4 \equiv \tau^8 + 2\tau^6 + 3\tau^4 + 2\tau^2 + 4\tau + 2$,\\
    $1^4 + (\tau+\tau^3+1)^4 \equiv \tau^{12} + 2\tau^6 + \tau^4 + 6\tau^2 + 4\tau +2$, \\
    $1^4 + (\tau^2+\tau^3+1)^4 \equiv \tau^8+2\tau^6+2\tau^4+4\tau^3+4\tau^2+2$,\\
    $1^4 + (\tau+\tau^2+\tau^3+1)^4 \equiv \tau^{12}+\tau^8+3\tau^4+4\tau^3+2\tau^2+4\tau+2$, \\
    $(\tau+1)^4 + (\tau^2+1)^4 \equiv \tau^8 + 7\tau^4 + 4\tau^3 + 2\tau^2 + 4\tau + 2$, \\
    $(\tau+1)^4 + (\tau^3+1)^4 \equiv \tau^{12} + 2\tau^6 + \tau^4 + 6\tau^2 + 4\tau + 2$, \\
    $(\tau+1)^4 + (\tau+\tau^2+1)^4 \equiv \tau^8 + 2\tau^6 + 4\tau^4 + 4\tau^3 + 2$, \\
    $(\tau+1)^4 + (\tau+\tau^3+1)^4 \equiv \tau^{12}+ 2\tau^6 + 2\tau^4 + 4\tau^3 + 4\tau^2 + 2$, \\
    $(\tau+1)^4 + (\tau^2+\tau^3+1)^4 \equiv \tau^8 + 2\tau^6 + 3\tau^4+2\tau^2+4\tau+2$, \\
    $(\tau+1)^4 + (\tau+\tau^2+\tau^3+1)^4 \equiv \tau^8 + 2$, \\
    $(\tau^2+1)^4 + (\tau^3+1)^4 \equiv \tau^{12}+ \tau^8 + 2\tau^6 + 6\tau^4 + 4\tau^3 + 4\tau^2 + 2$,\\
    $(\tau^2+1)^4 + (\tau+\tau^2+1)^4 \equiv 2\tau^8 + 2\tau^6 + \tau^4 + 6\tau^2 + 4\tau + 2$, \\
    $(\tau^2+1)^4 + (\tau+\tau^3+1)^4 \equiv \tau^{12}+\tau^8+2\tau^6+7\tau^4+2\tau^2+4\tau+2$, \\
    $(\tau^2+1)^4 + (\tau^2+\tau^3+1)^4 \equiv 2\tau^8 + 2\tau^6 + 4\tau^3 + 2$, \\
    $(\tau^2+1)^4 + (\tau+\tau^2+\tau^3+1)^4 \equiv \tau^4+4\tau^3+6\tau^2+4\tau+2$, \\
    $(\tau^3+1)^4 + (\tau+\tau^2+1)^4 \equiv \tau^{12}+\tau^8+3\tau^4+4\tau^3+2\tau^2+4\tau+2$, \\
    $(\tau^3+1)^4 + (\tau+\tau^3+1)^4 \equiv \tau^4+4\tau^3+6\tau^2+4\tau+2$, \\
    
    $(\tau^3+1)^4 + (\tau^2+\tau^3+1)^4 \equiv \tau^{12}+\tau^8+2\tau^4+4\tau^2+2$, \\

    $(\tau^3+1)^4 + (\tau+\tau^2+\tau^3+1)^4 \equiv \tau^8+2\tau^6+3\tau^4+2\tau^2+4\tau+2$, \\

    $(\tau+\tau^2+1)^4 + (\tau+\tau^3+1)^4 \equiv \tau^8+2$, \\
     
    $(\tau+\tau^2+1)^4 + (\tau^2+\tau^3+1)^4 \equiv \tau^4+4\tau^3+6\tau^2+4\tau+2$, \\
    $(\tau+\tau^2+1)^4 + (\tau+\tau^2+\tau^3+1)^4 \equiv \tau^{12}+2\tau^6+2\tau^4+4\tau^3+4\tau^2+2$, \\
         
    $(\tau+\tau^3+1)^4 + (\tau^2+\tau^3+1)^4 \equiv \tau^{12}+\tau^8+3\tau^4+4\tau^3+2\tau^2+4\tau+2$, \\
    $(\tau+\tau^3+1)^4 + (\tau+\tau^2+\tau^3+1)^4 \equiv \tau^{12}+\tau^8+2\tau^6+4\tau^3+2$, \\
    $(\tau^2+\tau^3+1)^4 + (\tau+\tau^2+\tau^3+1)^4 \equiv \tau^{12}+2\tau^6+\tau^4+6\tau^2+4\tau+2. $
\end{tabular}
\end{center}
After removing any duplicates we get the following remainders:
\begin{center}
\begin{tabular}{@{}l@{}}
    $\tau^4 + 4\tau^3 + 6\tau^2 + 4\tau +2$, \\
    $\tau^8 + 2\tau^6 + 3\tau^4 + 2\tau^2 + 4\tau + 2  $, \\
    $\tau^{12} + 2\tau^6 + \tau^4 + 6\tau^2 + 4\tau +2$, \\
    $\tau^{12}+\tau^8+3\tau^4+4\tau^3+2\tau^2+4\tau+2$, \\
    \\
    $\tau^{12} + 2\tau^6 + 4\tau^3 + 2$, \\
    $\tau^{12}+ 2\tau^6 + 2\tau^4 + 4\tau^3 + 4\tau^2 + 2$, \\
    $\tau^8 + 2$, \\
    $\tau^{12}+\tau^8+2\tau^4+4\tau^2+2$, \\
    \\
    $\tau^8+2\tau^6+2\tau^4+4\tau^3+4\tau^2+2$,  \\
    $\tau^{12}+\tau^8+2\tau^6+4\tau^3+2$.
\end{tabular}
\end{center}

Let us now add a fourth power divisible by $\mathfrak{p}$ to the above elements. First four remainders will not be 0 modulo $\mathfrak{p}^{13}$, because $\mathfrak{p}^6 \parallel 2\tau^2$ and any other summands which can possibly cancel out $2\tau^2$ in the combined sum of three fourth powers are $\tau^4$ and $2$, which are the only elements not divisible by $\mathfrak{p}^7$. However, when those two elements are added, they are divisible by $\mathfrak{p}^9$ by assumption. Consequently, the term $2\tau^2$ cannot be reduced in any way.

The next four remainders can vanish only if $\mathfrak{p}^{10} \parallel \tau^4+2$ or $\mathfrak{p}^{8} \parallel \tau^4+2$. To see this, note that they have an even number of $\tau^4$. To cancel out the 2, the fourth power divisible by $\mathfrak{p}$, which we are adding must then have an odd number of $\tau^4$. Hence it is either $\tau^4$ or $\tau^8+2\tau^6+\tau^4$. Now adding them and reducing the good sums of three fourth powers modulo $\mathfrak{p}^{11}$ gives the following: 
\[ 2\tau^6 + (\tau^4 +2), \  \tau^8 + (\tau^4 + 2), \ 2\tau^4 + (\tau^4+2), \ 2\tau^6 + (\tau^4+1)(\tau^4+2)\]
and they can be 0 only if $\mathfrak{p}^{10} \parallel \tau^4+2$ or $\mathfrak{p}^{8} \parallel \tau^4+2$. To simplify the above, we used the fact that $2\tau^6 + 4\tau^2 \equiv 0 \implies 2\tau^6 \equiv 4\tau^2 \pmod{\mathfrak{p}^{11}} $, which follows from $\mathfrak{p}^{10} \parallel 2\tau^6, 4\tau^2 $. 

The last two remainders can be 0 only if $\mathfrak{p}^{11} \parallel \tau^4+2$ or $\mathfrak{p}^{8} \parallel \tau^4+2$. Again, we get that the divisible by $\mathfrak{p}$ fourth power that we are adding must be either $\tau^4$ or $\tau^8+2\tau^6+\tau^4$. Hence we obtain following elements modulo $\mathfrak{p}^{12}$:
\[ 4\tau^3 + (\tau^4+2\tau^2+1)(\tau^4+2), \ 4\tau^3 + 4\tau^2 + 2\tau^4 + (\tau^4+2), \ 4\tau^3 + 2\tau^6 + \tau^8 + (\tau^4+2), \  4\tau^3 + (\tau^4+2),\]
so the result follows. 
\end{proof}
\begin{lemma}\label{2.4} If $\mathfrak{p}^7 \parallel \tau^4+2$ or $\mathfrak{p}^9 \parallel \tau^4+2$ then $s > 3$ as long as $\mathfrak{p}^{13} \nmid \tau^{12}+4\tau^3+2(\tau^4+2)$.
\end{lemma}

\begin{proof} By the previous lemmas, we already have $s>2$. To prove that $s > 3$ we need to consider all good sums of four fourth powers. Before we start the proof, let us make some reductions. Note that $2\tau^6 + 4\tau^2 \equiv 2\tau^2(\tau^4+2) \equiv 0 \pmod{\mathfrak{p}^{13}}$ and $0 \equiv (\tau^4+2)^2 \equiv \tau^8 + 4\tau^4+4 \pmod{\mathfrak{p}^{13}} \implies \tau^8+4 \equiv 4\tau^{4} \equiv\tau^{12}\pmod{\mathfrak{p}^{13}}$. Recall  now all the possible remainders given by a sum of two not divisible by $\mathfrak{p}$ fourth powers: 

$$
\left.
    \begin{array}{ll}
        \tau^4 + 4\tau^3 + 6\tau^2 + 4\tau +2, \\
        \tau^8 + 2\tau^6 + 3\tau^4 + 2\tau^2 + 4\tau + 2  , \\
        \tau^{12} + 2\tau^6 + \tau^4 + 6\tau^2 + 4\tau +2, \\
        \tau^{12}+\tau^8+3\tau^4+4\tau^3+2\tau^2+4\tau+2, 
    \end{array}
\right \}A 
$$
$$
\left.
    \begin{array}{ll}
        2,\\
        2\tau^4+4\tau^2+2,\\
        \tau^{12} + 2\tau^6 + 4\tau^3 + 2, \\
        \tau^{12}+ 2\tau^6 + 2\tau^4 + 4\tau^3 + 4\tau^2 + 2, \quad \ \   \\
        \tau^8 + 2, \\
        \tau^{12}+\tau^8+2\tau^4+4\tau^2+2, \\
        \tau^8+2\tau^6+2\tau^4+4\tau^3+4\tau^2+2,  \\
        \tau^{12}+\tau^8+2\tau^6+4\tau^3+2. 
    \end{array}
\right \}B
$$

They are split into two parts. Part A contains the remainders which have an odd number of $\tau^4$ and not divisible by 4 number of $\tau^2$. Part B has remainders with an even number of $\tau^4$ and divisible by 4 number of $\tau^2$. By assumption, we have $\mathfrak{p}^7 \mid \tau^4+2$, hence all the elements of $A$ are in $\mathfrak{p}^6\setminus \mathfrak{p}^7$ (because $\mathfrak{p}^6 \parallel 2\tau^2$) and all the elements of B are in $\mathfrak{p}^4\setminus \mathfrak{p}^5$ (because $\mathfrak{p}^4 \parallel 2$).

Below we list every possible remainder modulo $\mathfrak{p}^{13}$, which is achieved by summing two fourth powers divisible by $\mathfrak{p}$ (not necessarily distinct):
$$
\left.
    \begin{array}{ll}
        \tau^8, \quad 2\tau^4, \quad \tau^{12}+2\tau^4, \quad \tau^{12}, \quad \tau^{12}+\tau^8, \quad  \tau^8 + 2\tau^6 + 2\tau^4, \quad \tau^{12}+\tau^8+\tau^4, \\ 
     \tau^8 + \tau^4, \quad \tau^{12}+\tau^4, \quad \tau^{12}+2\tau^6 + \tau^4, \quad 2\tau^6 + \tau^4, \quad \tau^{12}+\tau^8+2\tau^6+\tau^4. 
    \end{array}
\right \}C
$$

All the above elements are congruent to $0$ or $\tau^4$ modulo $\mathfrak{p}^8$. Hence, after adding them to any of the remainders from $A$, we will get an element which is in either $\mathfrak{p}^6\setminus \mathfrak{p}^7$ or $\mathfrak{p}^4\setminus \mathfrak{p}^5$.

Let us now examine the set $B$. Because those remainders are in $\mathfrak{p}^4 \setminus \mathfrak{p}^5$, the remainder from $C$ which we are adding have to be in $\mathfrak{p}^4\setminus \mathfrak{p}^5$, for the total good sum to be congruent to 0 modulo $\mathfrak{p}^{13}$. This leaves us with six elements from $C$ -- those which have an odd number of $\tau^4$.

All the remainders from $B$, considered modulo $\mathfrak{p}^{10}$ are in $\{2, \ 2\tau^4+2,\ \tau^8+2\}$. We used here the fact that $\mathfrak{p}^{10} \mid \tau^8+2\tau^4 = \tau^4(\tau^4+2)$. Similarly, remainders from $C$ with an odd number of $\tau^4$ are congruent to either $\tau^4$ or $\tau^4+\tau^8$ modulo $\mathfrak{p}^{10}$. Hence, their sum modulo $\mathfrak{p}^{10}$ is in $\{\tau^4+2, \ \tau^8 + (\tau^4+2),\ 2\tau^4+(\tau^4+2),\ (\tau^4+1)(\tau^4+2) \}$ and none of them is congruent to 0 modulo $\mathfrak{p}^{10}$ under the assumptions of lemma. 

To finish the proof, it is enough to check sums of four not divisible by $\mathfrak{p}$ fourth powers modulo $\mathfrak{p}^{13}$. It is helpful to view such sum as a sum of two good sums of two fourth powers, as we know what are the possible remainders of good sums of two fourth powers. Let us see what happens when we add two same sums of two fourth powers. For this we will multiply by 2 the elements from the sets $A$ and $B$. By doing this we see that the set of possible remainders is $\{4, \  \tau^{12}+4,\  2\tau^4+4\tau^2+4\}$ and none of those remainders are congruent to 0 modulo $\mathfrak{p}^{13}$ (for the last one note that $\mathfrak{p}^{10} \not\parallel 2(\tau^4+2)$ by the assumption).

If we add an element from $A$ to an element from $B$, the resulting element will not be congruent to zero modulo  $\mathfrak{p}^{13}$. Elements from $A$ are in $\mathfrak{p}^6\setminus \mathfrak{p}^7$ and elements from $B$ are in $\mathfrak{p}^4\setminus \mathfrak{p}^5$, hence their sum is in $\mathfrak{p}^4\setminus \mathfrak{p}^5$.

Consider now sums of two different elements from $A$: 
\begin{center}
\begin{tabular}{@{}l@{}}
    $(\tau^4 + 4\tau^3 + 6\tau^2 + 4\tau +2) + (\tau^8 + 2\tau^6 + 3\tau^4 + 2\tau^2 + 4\tau + 2) \equiv  \tau^{12}+\tau^8+2\tau^6 + 4\tau^3+4$,\\
    $(\tau^4 + 4\tau^3 + 6\tau^2 + 4\tau +2) + (\tau^{12} + 2\tau^6 + \tau^4 + 6\tau^2 + 4\tau +2) \equiv \tau^{12} +2\tau^6 + 2\tau^4 + 4\tau^3+4\tau^2+4 $,\\
    $(\tau^4 + 4\tau^3 + 6\tau^2 + 4\tau +2) + (\tau^{12}+\tau^8+3\tau^4+4\tau^3+2\tau^2+4\tau+2) \equiv \tau^{8}+4 $,\\
    $(\tau^8 + 2\tau^6 + 3\tau^4 + 2\tau^2 + 4\tau + 2) + (\tau^{12} + 2\tau^6 + \tau^4 + 6\tau^2 + 4\tau +2) \equiv \tau^{8}+4 $,\\
    $(\tau^8 + 2\tau^6 + 3\tau^4 + 2\tau^2 + 4\tau + 2) + (\tau^{12}+\tau^8+3\tau^4+4\tau^3+2\tau^2+4\tau+2) \equiv \tau^{12} +2\tau^6 + 2\tau^4+4\tau^3+4\tau^2+4 $,\\
    $(\tau^{12} + 2\tau^6 + \tau^4 + 6\tau^2 + 4\tau +2) + (\tau^{12}+\tau^8+3\tau^4+4\tau^3+2\tau^2+4\tau+2) \equiv \tau^{12}+\tau^8+2\tau^6 + 4\tau^3+4 $.
\end{tabular}
\end{center}
After removing duplicates and using properties proved at the beginning of the proof to simplify expressions, we are left with remainders: $2\tau^6+4\tau^3,\  \tau^{12}+2\tau^4+4\tau^3+4, \ \tau^{12}.$
The elements $2\tau^6 + 4\tau^3$ and $\tau^{12}$ are not congruent to 0 modulo $\mathfrak{p}^{13}$, and the remainder $\tau^{12} + 2\tau^4+4\tau^3+4$ is assumed not to vanish modulo $\mathfrak{p}^{13}$. The last element can actually vanish sometimes, namely, it is zero modulo $\mathfrak{p}^{13}$ iff $\mathfrak{p}^{7} \parallel \tau^4+2$ and $\mathfrak{p}^{12} \parallel 4\tau^3 + 2(\tau^4+2)$. We are left with sums of two different elements from $B$:
\begin{center}
\begin{tabular}{@{}l@{}}
    $2 + (2\tau^4+4\tau^2+2) \equiv 2\tau^4+4\tau^2+4$\\
    $2 + (\tau^{12}+2\tau^6+4\tau^3+2) \equiv \tau^{12} + 2\tau^6 + 4\tau^3 + 4$, \\
    $2 + (\tau^{12}+ 2\tau^6 + 2\tau^4 + 4\tau^3 + 4\tau^2 + 2) \equiv \tau^{12} + 2\tau^6 +2\tau^4 + 4\tau^3  +4\tau^2+ 4$, \\
    $2 + (\tau^8 + 2) \equiv \tau^{8}+4$, \\
    $2 + (\tau^{12}+\tau^8+2\tau^4+4\tau^2+2) \equiv \tau^{12}+\tau^8+2\tau^4+4\tau^2+4$, \\
    $2 + (\tau^8+2\tau^6+2\tau^4+4\tau^3+4\tau^2+2) \equiv \tau^{8}+2\tau^6+2\tau^4+4\tau^3+4\tau^2+4$,  \\
    $2 + (\tau^{12}+\tau^8+2\tau^6+4\tau^3+2) \equiv \tau^{12}+\tau^8+2\tau^6+4\tau^3+4$,\\
    $(2\tau^4+4\tau^2+2) + (\tau^{12}+2\tau^6+4\tau^3+2) \equiv \tau^{12} + 2\tau^6+2\tau^4+ 4\tau^3 +4\tau^2+ 4$, \\
    $(2\tau^4+4\tau^2+2) + (\tau^{12}+ 2\tau^6 + 2\tau^4 + 4\tau^3 + 4\tau^2 + 2) \equiv 2\tau^6 + 4\tau^3 + 4$, \\
    $(2\tau^4+4\tau^2+2) + (\tau^8 + 2) \equiv \tau^{8} + 2\tau^4 + 4\tau^2+4 $, \\
    $(2\tau^4+4\tau^2+2) + (\tau^{12}+\tau^8+2\tau^4+4\tau^2+2) \equiv \tau^{8}+4$, \\
    $(2\tau^4+4\tau^2+2) + (\tau^8+2\tau^6+2\tau^4+4\tau^3+4\tau^2+2) \equiv \tau^{12}+\tau^8+2\tau^6+4\tau^3+4$,  \\
    $(2\tau^4+4\tau^2+2) + (\tau^{12}+\tau^8+2\tau^6+4\tau^3+2) \equiv \tau^{12}+\tau^8+2\tau^6+2\tau^4+4\tau^3+4\tau^2+4$,\\ 
    $(\tau^{12}+2\tau^6+4\tau^3+2) + (\tau^{12}+ 2\tau^6 + 2\tau^4 + 4\tau^3 + 4\tau^2 + 2) \equiv 2\tau^4+4\tau^2+4$,\\
    $(\tau^{12}+2\tau^6+4\tau^3+2) + (\tau^8 + 2) \equiv \tau^{12}+\tau^8+2\tau^6+4\tau^3+4$,\\
    $(\tau^{12}+2\tau^6+4\tau^3+2) + (\tau^{12}+\tau^8+2\tau^4+4\tau^2+2) \equiv \tau^{8}+2\tau^6+2\tau^4+4\tau^3+4\tau^2+4$,\\
    $(\tau^{12}+2\tau^6+4\tau^3+2) + (\tau^8+2\tau^6+2\tau^4+4\tau^3+4\tau^2+2) \equiv \tau^{12}+\tau^8+2\tau^4+4\tau^2+4$,\\
    $(\tau^{12}+2\tau^6+4\tau^3+2) + (\tau^{12}+\tau^8+2\tau^6+4\tau^3+2) \equiv \tau^{8}+4$,\\
    $(\tau^{12}+ 2\tau^6 + 2\tau^4 + 4\tau^3 + 4\tau^2 + 2) + (\tau^8 + 2) \equiv  \tau^{12}+\tau^8+2\tau^6+2\tau^4 + 4\tau^3+4\tau^2+4 $,\\
    $(\tau^{12}+ 2\tau^6 + 2\tau^4 + 4\tau^3 + 4\tau^2 + 2) + (\tau^{12}+\tau^8+2\tau^4+4\tau^2+2) \equiv \tau^{12}+\tau^8+ 2\tau^6+4\tau^3+4$,\\
    $(\tau^{12}+ 2\tau^6 + 2\tau^4 + 4\tau^3 + 4\tau^2 + 2) + (\tau^8+2\tau^6+2\tau^4+4\tau^3+4\tau^2+2) \equiv \tau^{8}+4$,\\
    $(\tau^{12}+ 2\tau^6 + 2\tau^4 + 4\tau^3 + 4\tau^2 + 2) + (\tau^{12}+\tau^8+2\tau^6+4\tau^3+2) \equiv \tau^{8}+2\tau^4+4\tau^2+4$,\\
    $(\tau^8 + 2) + (\tau^{12}+\tau^8+2\tau^4+4\tau^2+2) \equiv 2\tau^4+4\tau^2+4$, \\
    $(\tau^8 + 2) + (\tau^8+2\tau^6+2\tau^4+4\tau^3+4\tau^2+2) \equiv \tau^{12}+2\tau^6+2\tau^4+4\tau^3+4\tau^2+4$,  \\
    $(\tau^8 + 2) + (\tau^{12}+\tau^8+2\tau^6+4\tau^3+2) \equiv 2\tau^6+4\tau^3+4$,\\
    $(\tau^{12}+\tau^8+2\tau^4+4\tau^2+2) + (\tau^8+2\tau^6+2\tau^4+4\tau^3+4\tau^2+2) \equiv \tau^{12}+2\tau^6+4\tau^3+4$,\\
    $(\tau^{12}+\tau^8+2\tau^4+4\tau^2+2) + (\tau^{12}+\tau^8+2\tau^6+4\tau^3+2) \equiv \tau^{12}+2\tau^6+2\tau^4+4\tau^3+4\tau^2+4$,\\
    $(\tau^8+2\tau^6+2\tau^4+4\tau^3+4\tau^2+2) + (\tau^{12}+\tau^8+2\tau^6+4\tau^3+2) \equiv 2\tau^4+4\tau^2+4$.
    
\end{tabular}
\end{center}
After removing any duplicates and simplifying as before, we are left with following elements:

\begin{center}
\begin{tabular}{@{}l@{}}
    $2\tau^4+4\tau^2+4, \quad \tau^{12}+2\tau^6+4\tau^3+4, \quad 
    \tau^{12}+2\tau^4+4\tau^3+4, \quad \tau^{12}, \quad 2\tau^4+4\tau^2,$\\
    $\tau^{12}+2\tau^4+4\tau^3, \quad 2\tau^6+4\tau^3, \quad
    2\tau^6+4\tau^3+4, \quad \tau^{12}+2\tau^4+4\tau^2, \quad 2\tau^4+4\tau^3.$
\end{tabular}
\end{center}
We see that the only remainder which can vanish modulo $\mathfrak{p}^{13}$ is the already mentioned element $\tau^{12} + 4\tau^3+2(\tau^4+2)$. This finishes the proof.
\end{proof}

We may now state the main result of this section.
\begin{thm}\label{main}
Let $K$ be a number field, with a prime ideal $\p \mid (2)$ of $\OO_k$ with $e(\p) = 4, \ f(\p)=1$. Let $\tau$ be any element in $\mathfrak{p}\setminus \mathfrak{p}^2$. Then $s = s_4(K_{\mathfrak{p}})$ is one of $\{1,2,3,4,5,6\}$, specifically
\end{thm}

\noindent
$ \text{if } \mathfrak{p}^5 \parallel \tau^4+2 \text{ then } s = \begin{cases}
    \textbf{3} \text{ iff } \mathfrak{p}^{13} \mid \tau^{12}+\tau^8+2\tau^6+4\tau^3+4\\
    \textbf{4} \text{ iff } \mathfrak{p}^{13} \mid \tau^8+2\tau^6+4\tau^3+4\\
    \textbf{5} \text{ iff } \mathfrak{p}^{13} \mid  \tau^8+4\tau^3+4\tau^2+4\\
    \textbf{6} \text{ iff } \mathfrak{p}^{13} \mid  \tau^{12}+\tau^8+4\tau^3+4\tau^2+4
\end{cases} $

\noindent
$\text{if } \mathfrak{p}^6 \parallel \tau^4+2 \text{ then } s = \begin{cases}
    \textbf{1} \text{ iff at least one of } \{A_1,A_2,A_3,A_4\} \text{ is divisible by } \mathfrak{p}^{13} \\
    \textbf{2} \text{ iff at least one of } \{A_1,A_2,A_3,A_4\}  \text{ is in } \mathfrak{p}^{12}\setminus\p^{13} \\
    \text{or at least one of } \{A_1+\tau^8, \ A_2+\tau^8, \ A_3+\tau^8, \ A_4+\tau^8\} \text{ is divisible by } \mathfrak{p}^{12} \\
    \textbf{3} \text{ otherwise}
\end{cases}$
\begin{center}
where:
\begin{tabular}{@{}l@{}}
    $A_1 = 4\tau^3+4\tau^2+4\tau+(\tau^4+2\tau^2+2),$\\
    $A_2 = 4\tau + (\tau^4+1)(\tau^4+2\tau^2+2),$\\
    $A_3 = 4\tau + (\tau^4+2\tau^2+2), $\\
    $A_4= 4\tau^3+4\tau^2+4\tau+(\tau^4+1)(\tau^4+2\tau^2+2). $
\end{tabular}
\end{center}
$\text{if } \mathfrak{p}^7 \parallel \tau^4+2 \text{ then } s = \begin{cases}
    \textbf{3} \text{ iff } \mathfrak{p}^{13} \mid \tau^{12}+4\tau^3+2(\tau^4+2)\\
    \textbf{4} \text{ otherwise}
\end{cases}$ 
\\
$\text{if } \mathfrak{p}^8 \parallel \tau^4+2 \text{ then } s = \begin{cases}
    \textbf{2} \text{ iff at least one of } \{ (\tau^8+\tau^4+2), \  4\tau^2+ (\tau^8+\tau^4+2), \\
    \tau^{12}+4\tau^3+(\tau^8+\tau^4+2), \ \tau^{12}+4\tau^3+4\tau^2+(\tau^8+\tau^4+2)\} \text{ is divisible by } \mathfrak{p}^{13}\\
    \textbf{3} \text{ iff at least one of } \{ \tau^{12}+(\tau^8+\tau^4+2), \ \tau^{12} + 4\tau^2+ (\tau^8+\tau^4+2), \\
    4\tau^3+(\tau^8+\tau^4+2), \ 4\tau^3+4\tau^2+(\tau^8+\tau^4+2)\} \text{ is divisible by } \mathfrak{p}^{13}\\
    \textbf{4} \text{ iff } \mathfrak{p}^9 \parallel \tau^8+\tau^4+2
\end{cases}$
\\
$\text{if } \mathfrak{p}^9 \parallel \tau^4+2 \text{ then } s = \textbf{4}$
\\
$\text{if } \mathfrak{p}^{10} \parallel \tau^4+2 \text{ then } s = \begin{cases}
    \textbf{2} \text{ iff } \mathfrak{p}^{13} \mid \tau^{12}+4\tau^3+4\tau^2+\tau^4+2 \text{ or } \tau^{12}+4\tau^2+\tau^4+2 \\
    \textbf{3} \text{ otherwise}
\end{cases}$
\\
$\text{if } \mathfrak{p}^{11} \parallel \tau^4+2 \text{ then } s = \begin{cases}
    \textbf{2} \text{ iff } \mathfrak{p}^{13} \mid 4\tau^3+\tau^4+2  \\
    \textbf{3} \text{ otherwise}
\end{cases}$
\\
$\text{if } \mathfrak{p}^{12} \parallel \tau^4+2 \text{ then } s = \textbf{3}$
\\
$\text{if } \mathfrak{p}^{13} \mid \tau^4+2 \text{ then } s = \textbf{2}.$

\begin{remark}
    For any number field with a prime dividing $(2)$ and with $e=4, \ f=1$ exactly one of the above cases (and subcases) will be true and will give the value of $s$, however we do not know if all the cases listed above will exist for some number field. We also emphasize that the above theorem does not depend on the choice of $\tau$.
\end{remark}

\begin{proof}
Proof will be a case by case analysis.

\textbf{Case \ref{cef}.5.} $\mathfrak{p}^5 \parallel \tau^4+2$.

By Lemma \ref{2.2} we know that $s \geq 3$. We have previously mentioned that a sum of fourth powers must have an even number of not divisible by $\mathfrak{p}$ fourth powers to be congruent to 0 modulo $\mathfrak{p}^{13}$. If the number of fourth powers not divisible by $\mathfrak{p}$ is even and not divisible by 4, then such sum is equal either to $2$ or $\tau^4+2$ modulo $\mathfrak{p}^6$. Both of those elements are nonzero modulo $\mathfrak{p}^6$, hence cannot vanish modulo $\mathfrak{p}^{13}$. As a consequence, we will consider good sums with the number of not divisible by $\mathfrak{p}$ summands equal to 4.

 In the previous lemma we have computed all possible remainders of such sums modulo $\mathfrak{p}^{13}$, except for sums of the form $A+B$. Any element of the aforementioned sum will have an element $2\tau^2 \in \mathfrak{p}^6\setminus \mathfrak{p}^7 $. The other terms which are not divisible by $\mathfrak{p}^7$
 are $\tau^4$ and $2$, but $ \mathfrak{p}^5 \parallel \tau^4+2$. Hence, such sum cannot vanish modulo $\mathfrak{p}^{13}$.
 
Let us recall all possible sums of four fourth powers modulo $\p^{13}$ of elements not divisible by $\p$, except for elements of the form $A+B$ of interest to us:
$$
\left.
    \begin{array}{ll}
        4,\\
        \tau^{12}+4,\\
        2\tau^4+4\tau^2+4,\\
        \tau^{12} + 2\tau^6 + 4\tau^3 + 4, \\
        \tau^{12}+ 2\tau^6 + 2\tau^4 + 4\tau^3 + 4\tau^2 + 4, \\
        \tau^8 + 4, \\
        \tau^{12}+\tau^8+2\tau^4+4\tau^2+4, \\
        \tau^8+2\tau^6+2\tau^4+4\tau^3+4\tau^2+4,  \\
        \tau^{12}+\tau^8+2\tau^6+4\tau^3+4,\\
        2\tau^6 + 4\tau^3 + 4, \\
        \tau^8 + 2\tau^4 + 4\tau^2 + 4, \\
        \tau^{12}+\tau^8+2\tau^6+2\tau^4+4\tau^3+4\tau^2+4  
    \end{array}
\right \}D
$$

We have $\mathfrak{p}^9 \parallel 2(\tau^4+2) = 2\tau^4+4$ and $\mathfrak{p}^{10} \parallel (\tau^4+2)^2 = \tau^8 + 4\tau^4 + 4 \implies \mathfrak{p}^{10} \parallel \tau^8+4 $. Simple calculations shows that the only element from the above list which can be congruent to 0 modulo $\mathfrak{p}^{13}$ is $\tau^{12}+\tau^8+2\tau^6+4\tau^3+4$. Hence, we have $s=3$ if and only if $\mathfrak{p}^{13} \mid \tau^{12}+\tau^8+2\tau^6+4\tau^3+4 $.

As a next step, we would like to focus on the situation $s>3$. Let us begin with good sums with 5 summands where we add one fourth powers divisible by $\mathfrak{p}$ to the four summands of the form $A+B$. By a similar discussion as before, no such sum can vanish modulo $\mathfrak{p}^{13}$.

 Consider now adding a single fourth power divisible by $\p$ to the elements of the set $D$. If this fourth power has an odd number of $\tau^4$ then such element cannot vanish modulo $\p^{13}$. This is shown by a simple calculation. Hence, we are left with the following set of fourth powers $\{\tau^8, \tau^{12}+\tau^8, \tau^{12}\}$. The only element of $D$ which can be divisible by $\p^{12}$ is $\tau^{12}+\tau^8+2\tau^6+4\tau^3+4$. Hence, if $\mathfrak{p}^{12} \parallel \tau^{12}+\tau^8+2\tau^6+4\tau^3+4$ then after adding $\tau^{12}$ we indeed obtain sum of 5 fourth powers which vanish modulo $\p^{13}$. Note that the above is equivalent to $\mathfrak{p}^{13} \mid \tau^8+2\tau^6+4\tau^3+4$.

 If we choose to add $\tau^8$ or $ \tau^{12}+\tau^8$ then the possible candidates from $D$ have to have and odd number of summands from $\p^8\setminus \p^9$. If such element has 3 summands in $\p^8\setminus \p^9$, namely $\tau^8, \ 2\tau^4, \ 4$ , then after adding $\tau^8$ we have $2\tau^8 \equiv \tau^{12} \pmod{\mathfrak{p}^{13}}$ and $\mathfrak{p}^9 \parallel 2\tau^4+4$. But now, such sum does not posses other elements belonging to $\p^9\setminus \p^{10}$, hence is nonzero modulo $\p^{10}$.

In order to check all possibilities, we need to add $\tau^8$ or $ \tau^{12}+\tau^8$ to element from $D$ with only one summand in $\p^8\setminus \p^9$. There four such elements, $4,\  \tau^{12}+4,\ \tau^{12}+2\tau^6+4\tau^3+4,\  2\tau^6+4\tau^3+4$. Because $\mathfrak{p}^{10} \parallel \tau^8+4$ we are left with $\tau^{12}+2\tau^6+4\tau^3+4$ and $2\tau^6+4\tau^3+4$. After adding $\tau^8$ or $\tau^{12}+\tau^8$  to them we get either $\tau^{12}+\tau^8+2\tau^6+4\tau^3+4$ or $\tau^8+2\tau^6+4\tau^3+4$. If the remainder $\tau^{12}+\tau^8+2\tau^6+4\tau^3+4$ is congruent to 0 modulo $\mathfrak{p}^{13}$, then $s=3$, by the previous discussion. Hence, $s=4$ if and only if $\mathfrak{p}^{13} \mid \tau^8+2\tau^6+4\tau^3+4 $.

To determine when $s=5$, we will add sums of two fourth powers divisible by $\p$ to elements from $D$. Similarly as before, we only need to consider elements with even number of $\tau^4$. We get the following
\[\tau^8,    \quad \tau^{12}+\tau^8, \quad \tau^{12}, \quad 2\tau^4, \quad \tau^{12}+2\tau^4, \quad \tau^8+2\tau^6+2\tau^4.\]
First three elements can be achieved by adding a single fourth powers, hence we omit them. Straightforward calculation shows that modulo $\p^{13}$ we able to obtain following elements
\[ \tau^{12}+\tau^8+4\tau^2+4, \quad \tau^8+4\tau^3+4\tau^2+4, \quad \tau^8+4\tau^2+4.\]

If $s>4$ then $\mathfrak{p}^{13} \nmid \tau^{12}+\tau^8+2\tau^6+4\tau^3+4 $ and $\mathfrak{p}^{13} \nmid \tau^8+2\tau^6+4\tau^3+4$ which implies $\mathfrak{p}^{12} \nmid \tau^8+2\tau^6+4\tau^3+4 $. Because $\mathfrak{p}^{10} \parallel \tau^8+4$ for the former to be true we have to have $\mathfrak{p}^{12} \mid \tau^8+2\tau^6+4$. Since $\mathfrak{p}^{11} \parallel 2\tau^2(\tau^4+2)  = 2\tau^6+4\tau^2$ we get $\mathfrak{p}^{11} \parallel \tau^8+2\tau^6+4  + (2\tau^6+4\tau^2) \equiv  \tau^8+4\tau^2+4  $ modulo $\mathfrak{p}^{13}$, thus the only remainder that can be congruent to 0 in this case is $ \tau^8+4\tau^3+4\tau^2+4$. Hence, $s=5$ if and only if $\mathfrak{p}^{13} \mid  \tau^8+4\tau^3+4\tau^2+4$.

If $s > 5$ then $\tau^8+4\tau^3+4\tau^2+4$ does not vanish modulo $\mathfrak{p}^{13}$, and we have still have $\mathfrak{p}^{12} \mid \tau^8+2\tau^6+4$, since $s>4$. We have $\mathfrak{p}^{11} \parallel 2\tau^6+4\tau^2 = 2\tau^2(\tau^4+2)$, and $\mathfrak{p}^{11} \parallel \tau^8+4\tau^2+4$, which combined with $\mathfrak{p}^{13} \nmid  \tau^8+4\tau^3+4\tau^2+4$ gives $\mathfrak{p}^{12} \parallel \tau^8+4\tau^3+4\tau^2+4 $. Furthermore, adding $(\tau^3)^4$ to this sum gives a good sum of seven fourth powers which is 0 modulo $\mathfrak{p}^{13}$, hence the fourth level is equal to 6. This finishes proof of the first case.

\end{proof}

%Now the time has come to split our considerations into nine cases of $\alpha$ for which $\mathfrak{p}^{\alpha} \parallel \tau^4+2 $. In the rest of this section, we will compute the fourth level in each case separately.

\textbf{Case \ref{cef}.6.} $\mathfrak{p}^6 \parallel \tau^4+2$.

Note that $\tau^{12} \equiv (\tau^4+2)^2 \equiv \tau^8 + 4\tau^4+4 \pmod{\mathfrak{p}^{13}}$ implies $\tau^8+4 \equiv 0 \pmod{\mathfrak{p}^{13}}$. Hence
\[ (\tau+1)^4+(\tau+\tau^2+\tau^3+1)^4+2 \cdot 1^4 \equiv (\tau^4+4\tau^3+6\tau^2+4\tau+1) + \] 
\[(\tau^{12}+\tau^8+3\tau^4+4\tau^3+2\tau^2+4\tau+1) + 2 \equiv \tau^8+4 \equiv 0 \pmod{\mathfrak{p}^{13}}. \]
This implies $s \leq 3$. Consider now sums of two fourth powers not divisible by $\p$. If such sum vanish modulo $\p^{13}$, then it has to have a summand in $\p^6\setminus \p^7$ since $\tau^4+2 \in \p^6\setminus \p^7$. By Lemma \ref{2.4} there are four such sums i.e. the ones from the set $A$. Using $2\tau^6+4\tau^2 \equiv 2\tau^2(\tau^4+2) \equiv \tau^{12} \pmod{\mathfrak{p}^{13}}$ we can transform them into the following:  
\begin{center}
\begin{tabular}{@{}l@{}}
    $A_1 = 4\tau^3+4\tau^2+4\tau+(\tau^4+2\tau^2+2),$\\
    $A_2 = 4\tau + (\tau^4+1)(\tau^4+2\tau^2+2),$\\
    $A_3 = 4\tau + (\tau^4+2\tau^2+2), $\\
    $A_4= 4\tau^3+4\tau^2+4\tau+(\tau^4+1)(\tau^4+2\tau^2+2). $
\end{tabular}
\end{center}
Hence, $s=1$ if and only if at least one of the above elements is congruent to 0 modulo $\mathfrak{p}^{13}$.

Consider now the good sums having 3 elements. Now no sum of the form: element from $B$ with a fourth power divisible by $\p$ can vanish modulo $\p^{13}$, because no sum in $B$ has a summand in $\p^6 \setminus \p^7$, hence $\tau^4+2$ cannot be reduced. Now we will consider adding a fourth power divisible by $\p$ to the sums $A_1,A_2,A_3,A_4$. Similarly as before, fourth power divisible by $\p$ has to have an even number of $\tau^4$. This leaves us with three possible fourth powers $\{\tau^8, \tau^{12}+\tau^8, \tau^{12}\}$.
Again, by considering cases, simple computation shows that such a sum is congruent to zero modulo $\p^{13}$ if either for some $i=1,2,3,4$ $A_i \in \p^{12}\setminus \p^{13}$ or for some $i=1,2,3,4$ $A_i+\tau^8$ is divisible by $\p^{12}$. The proof of this case is done.

\textbf{Case \ref{cef}.7.} $\mathfrak{p}^7 \parallel \tau^4+2$. We have:
\[ (\tau^3+\tau^2)^4+4\cdot1^4 \equiv \tau^{12}+\tau^8+4 \equiv \tau^8+4\tau^4+4 \equiv (\tau^4+2)^2 \equiv 0 \pmod{\mathfrak{p}^{13}}\]
thus $s \leq 4$. Application of Lemma \ref{2.4} finishes the proof.

\textbf{Case \ref{cef}.8.} $\mathfrak{p}^8 \parallel \tau^4+2$. 

By construction from the case \ref{cef}.7 and Lemma \ref{2.2} we get $2 \leq s \leq 4$. In this case we will make use of the sets $A$ and $B$. In order to obtain $s=2$ we have to add a fourth power divisible by $\p$ to elements from $A$ or $B$. If we add such an element to element from $A$, then the resulting sum will be either in $\p^4\setminus \p^5$ or $\p^6\setminus \p^7$. Again, every element from $B$ has an even number of $\tau^4$, hence the fourth power we are adding has to have an odd number of $\tau^4$ and this leaves only $\tau^4$ or $\tau^8+2\tau^6+\tau^4$. Notice that $2\tau^4+4 \equiv \tau^{12} \equiv 2\tau^8 \equiv \tau^8+2\tau^4 \pmod{\mathfrak{p}^{13}} \implies \tau^8 \equiv 2\tau^4 \equiv 4$ as well as $2\tau^6 \equiv 4\tau^2 \pmod{\mathfrak{p}^{13}}$. Using those properties we can simplify the elements of $B$:  
\begin{center}
\begin{tabular}{r c l}
$2$ & $\equiv$ & $ 2$,\\
$2\tau^4+4\tau^2+2$ & $\equiv$ & $ \tau^8+4\tau^2+2$\\
$\tau^{12} + 2\tau^6 + 4\tau^3 + 2$ & $\equiv$ & $ \tau^{12}+4\tau^3+4\tau^2+2 $\\
$\tau^{12}+ 2\tau^6 + 2\tau^4 + 4\tau^3 + 4\tau^2 + 2$ & $\equiv$ & $ \tau^{12}+\tau^8+4\tau^3+2 $ \qquad \qquad \qquad \ \\
$\tau^8 + 2$ & $\equiv$ & $ \tau^8+2 $\\
$\tau^{12}+\tau^8+2\tau^4+4\tau^2+2$ & $\equiv$ & $ 4\tau^2+2$\\
$\tau^8+2\tau^6+2\tau^4+4\tau^3+4\tau^2+2$ & $\equiv$ & $ 4\tau^3+2$\\
$\tau^{12}+\tau^8+2\tau^6+4\tau^3+2$ & $\equiv$ & $ \tau^{12}+\tau^8+4\tau^3+4\tau^2+2$
\end{tabular}
\end{center}
When we add either $\tau^4$ or $\tau^8+2\tau^6+\tau^4$ to each of the above remainders, we see that the only elements which can possibly vanish modulo $\mathfrak{p}^{13}$ are the following:
\[(\tau^8+\tau^4+2), \quad 4\tau^2+ (\tau^8+\tau^4+2), \quad \tau^{12}+4\tau^3+(\tau^8+\tau^4+2), \quad \tau^{12}+4\tau^3+4\tau^2 + (\tau^8+\tau^4+2).\]
Hence,  $s=2$ if and only if at least one remainder from the above set is 0 modulo $\mathfrak{p}^{13}$. 

%Consider now good sums of four fourth powers.
 Using the properties $\tau^8\equiv 2\tau^4 \equiv 4$ and $2\tau^6 \equiv 4\tau^2$  modulo $\mathfrak{p}^{13}$ let us rewrite the elements of $A$:
\begin{center}
\begin{tabular}{r c l}
    $\tau^4 + 4\tau^3 + 6\tau^2 + 4\tau +2$ & $\equiv$ & $\tau^4 + 4\tau^3 + 6\tau^2 + 4\tau +2$  \\
   $\tau^8 + 2\tau^6 + 3\tau^4 + 2\tau^2 + 4\tau + 2$ & $\equiv$ & $\tau^{12}  + \tau^4 + 6\tau^2 + 4\tau + 2$ \\
    $\tau^{12} + 2\tau^6 + \tau^4 + 6\tau^2 + 4\tau +2$ & $\equiv$ & $\tau^{12}  + \tau^4 + 2\tau^2 + 4\tau +2$ \qquad \qquad \quad \quad \\
    $\tau^{12}+\tau^8+3\tau^4+4\tau^3+2\tau^2+4\tau+2$ & $\equiv$ & $\tau^4+4\tau^3+2\tau^2+4\tau+2$
\end{tabular}
\end{center}

Consider now good sums of four fourth powers. Clearly, every such sum has to have 4 as a summand. Even more,  every such sum which can be congruent to 0 modulo $\mathfrak{p}^{13}$ is also congruent to $0 \equiv \tau^{12} + 2\tau^8 \equiv \tau^{12} + \tau^8 + 4 \pmod{\mathfrak{p}^{13}}$. Sums of four fourth powers not divisible by $\p$ we already considered in Lemma \ref{2.4}, and there is no sum of the form $\tau^{12} + \tau^8 + 4$. As a consequence, no such sum can vanish modulo $\p^{13}$.

Hence, we have to consider a good sum of four fourth powers consisting of two elements not divisible by $\p$ and two divisible. Again, a rather tedious calculation shows that no element of the form $A+C$ can vanish modulo $\p^{13}$. We are left with considering elements of the form $B+C$. Element from $C$ which we are adding has to have an odd number of $\tau^4$. Let us recall the possible remainders from the set $C$.
\[\tau^8 + \tau^4, \quad \tau^{12}+\tau^8+\tau^4, \quad \tau^{12}+\tau^4, \quad \tau^{12}+2\tau^6 + \tau^4, \quad 2\tau^6 + \tau^4, \quad \tau^{12}+\tau^8+2\tau^6+\tau^4.\]

Some of the above elements have a summand in $\mathfrak{p}^8\setminus\p^9$, similarly as elements of $B$. If we add two remainders with parity of summands in $\p^8\setminus\p^9$, the resulting sum cannot vanish because of the presence of $\tau^4+2$ which is in  $\p^8\setminus\p^9$. We list the remaining possibilities below.
\begin{center}
\begin{tabular}{@{}l@{}}
    $2 + (\tau^8+\tau^4) \equiv (\tau^8+\tau^4+2)$,\\
    $2 + (\tau^{12}+\tau^8+\tau^4) \equiv \tau^{12}+(\tau^8+\tau^4+2)$,\\
    $2 + ( \tau^{12}+\tau^8+2\tau^6+\tau^4) \equiv \tau^{12}+4\tau^2+(\tau^8+\tau^4+2)$,\\
    $(\tau^8+4\tau^2+2) + (\tau^{12}+\tau^4) \equiv \tau^{12}+4\tau^2+ (\tau^8+\tau^4+2)$,\\
    $(\tau^8+4\tau^2+2) + (\tau^{12}+2\tau^6+\tau^4) \equiv \tau^{12}+(\tau^8+\tau^4+2)$,\\
    $(\tau^8+4\tau^2+2) + (2\tau^6+\tau^4) \equiv (\tau^8+\tau^4+2)$,\\
    $(\tau^{12}+4\tau^3+4\tau^2+2)+ (\tau^8+\tau^4) \equiv \tau^{12}+4\tau^3+4\tau^2+(\tau^8+\tau^4+2) $,\\
    $(\tau^{12}+4\tau^3+4\tau^2+2)+(\tau^{12}+\tau^8+\tau^4) \equiv 4\tau^3+4\tau^2+(\tau^8+\tau^4+2) $,\\
    $(\tau^{12}+4\tau^3+4\tau^2+2)+( \tau^{12}+\tau^8+2\tau^6+\tau^4) \equiv 4\tau^3+(\tau^8+\tau^4+2) $,\\
    $(\tau^{12}+\tau^8+4\tau^3+2) + (\tau^{12}+\tau^4) \equiv 4\tau^3 + (\tau^8+\tau^4+2)  $,\\
    $(\tau^{12}+\tau^8+4\tau^3+2) + (\tau^{12}+2\tau^6+\tau^4) \equiv 4\tau^3+4\tau^2+(\tau^8+\tau^4+2)$,\\
    $(\tau^{12}+\tau^8+4\tau^3+2) + (2\tau^6+\tau^4) \equiv \tau^{12}+4\tau^3+4\tau^2+(\tau^8+\tau^4+2)$ ,\\
    $(\tau^8+2) + (\tau^{12}+\tau^4) \equiv \tau^{12}+(\tau^8+\tau^4+2) $, \\
    $(\tau^8+2) + (\tau^{12}+2\tau^6+\tau^4) \equiv \tau^{12}+4\tau^2+(\tau^8+\tau^4+2)$, \\
    $(\tau^8+2) + (2\tau^6+\tau^4) \equiv 4\tau^2+(\tau^8+\tau^4+2)$, \\
    $(4\tau^2+2) + (\tau^8+\tau^4) \equiv 4\tau^2+(\tau^8+\tau^4+2) $, \\
    $(4\tau^2+2) + (\tau^{12}+\tau^8+\tau^4) \equiv \tau^{12}+4\tau^2+(\tau^8+\tau^4+2) $, \\

    $(4\tau^2+2) + (\tau^{12}+\tau^8+2\tau^6+\tau^4) \equiv \tau^{12}+(\tau^8+\tau^4+2) $, \\
    $(4\tau^3+2) + (\tau^8+\tau^4) \equiv 4\tau^3+(\tau^8+\tau^4+2) $, \\
    $(4\tau^3+2) + (\tau^{12}+\tau^8+\tau^4) \equiv \tau^{12}+4\tau^3+(\tau^8+\tau^4+2) $, \\
    
    $(4\tau^3+2) + (\tau^{12}+\tau^8+2\tau^6+\tau^4) \equiv \tau^{12}+4\tau^3+4\tau^2+(\tau^8+\tau^4+2) $, \\
    $(\tau^{12}+\tau^8+4\tau^3+4\tau^2+2) + (\tau^{12}+\tau^4) \equiv 4\tau^3+4\tau^2 +(\tau^8+\tau^4+2)$, \\
    $(\tau^{12}+\tau^8+4\tau^3+4\tau^2+2) + (\tau^{12}+2\tau^6+\tau^4)  \equiv 4\tau^3+(\tau^8+\tau^4+2) $, \\
    $(\tau^{12}+\tau^8+4\tau^3+4\tau^2+2) + (2\tau^6+\tau^4)  \equiv \tau^{12} + 4\tau^3 + (\tau^8+\tau^4+2) $, \\
\end{tabular}
\end{center}
After removing any duplicates and those remainders which are zero if $s=2$, we obtain that $s=3$ if and only if one of the following element is congruent to 0 modulo $\mathfrak{p}^{13}$:
\[ \tau^{12}+ (\tau^8+\tau^4+2), \quad \tau^{12}+4\tau^2 + (\tau^8+\tau^4+2), \quad 4\tau^3+(\tau^8+\tau^4+2), \quad 4\tau^3+4\tau^2
+ (\tau^8+\tau^4+2).\]
One can easily check that if one of the above elements vanishes, then $s >2$. Also, if $s>3$, then all of the above four remainders and all of the previous four remainders from case $s=2$, are nonzero modulo $\mathfrak{p}^{13}$. This implies that $\mathfrak{p}^9 \parallel \tau^8+\tau^4+2$ iff $s=4$. Let us explain that.

Of course, $\mathfrak{p}^9 \mid \tau^8+\tau^4+2$. If $\mathfrak{p}^{10} \parallel \tau^8+\tau^4+2$ then $\p^{11} \mid 4\tau^2 +\tau^8+\tau^4+2$. If $\p^{11} \parallel4\tau^2 +\tau^8+\tau^4+2$ then the element $ 4\tau^3+4\tau^2 +\tau^8+\tau^4+2$ is either in $\p^{12}\setminus\p ^{13}$ or in $\p^{13}$, in both cases we get $s=3$. If $\p^{12} \parallel4\tau^2 +\tau^8+\tau^4+2$ then the element $\tau^{12}+4\tau^2 + (\tau^8+\tau^4+2)$ vanishes hence $s=3$. If $\p^{13} \parallel4\tau^2 +\tau^8+\tau^4+2$ then $s=2$.

If $\mathfrak{p}^{11} \parallel \tau^8+\tau^4+2$ then either $4\tau^3+(\tau^8+\tau^4+2)$ or $\tau^{12}+4\tau^3+(\tau^8+\tau^4+2)$ vanishes modulo $\p^{13}$ and $s=3$ or $s=2$ respectively.

If $\mathfrak{p}^{12} \parallel \tau^8+\tau^4+2$ the $\tau^{12}+ (\tau^8+\tau^4+2)$
is congruent to zero and $s=3$.

If $\mathfrak{p}^{13} \mid \tau^8+\tau^4+2$ then clearly $s=2$. This finishes the proof of this case.
% This all gives: 
%\begin{center}
  %  $\text{if } \mathfrak{p}^8 \parallel \tau^4+2 \text{ then } s = \begin{cases}
   %     \textbf{2} \text{ iff at least one of } \{ (\tau^8+\tau^4+2), \  4\tau^2+ (\tau^8+\tau^4+2), \\
  %      \tau^{12}+4\tau^3+(\tau^8+\tau^4+2), \ \tau^{12}+4\tau^3+4\tau^2+(\tau^8+\tau^4+2)\} \text{ is divisible by } \mathfrak{p}^{13}\\
 %       \textbf{3} \text{ iff at least one of } \{ \tau^{12}+(\tau^8+\tau^4+2), \ \tau^{12} + 4\tau^2+ (\tau^8+\tau^4+2), \\
 %       4\tau^3+(\tau^8+\tau^4+2), \ 4\tau^3+4\tau^2+(\tau^8+\tau^4+2)\} \text{ is divisible by } \mathfrak{p}^{13}\\
  %      \textbf{4} \text{ iff } \mathfrak{p}^9 \parallel \tau^8+\tau^4+2.
 %   \end{cases}$
%\end{center}

\textbf{Case \ref{cef}.9.} $\mathfrak{p}^9 \parallel \tau^4+2$. 
This case follow directly from the construction in the case \ref{cef}.7 and Lemma \ref{2.4}.

\textbf{Case \ref{cef}.10.} $\mathfrak{p}^{10} \parallel \tau^4+2$. By Lemma \ref{2.2} we have $s \geq 2$. In Lemma \ref{2.3} we considered good sums having 3 elements. Using $\tau^8+2\tau^4 = \tau^4(\tau^4+2) \equiv 0$ and $2\tau^6 \equiv 4\tau^2 \pmod{\mathfrak{p}^{13}}$ we recall elements that possibly can vanish modulo $\p^{13}$:
\[\tau^{12}+4\tau^3+4\tau^2+\tau^4+2, \quad \tau^{12}+4\tau^2+\tau^4+2.\]
Hence, $s=2$ if and only if one of the above remainders vanish modulo $\p^{13}$.

If $s>2$ we have $\mathfrak{p}^{12} \not\parallel 4\tau^2+\tau^4+2$ and $\mathfrak{p}^{12} \not\parallel 4\tau^3+4\tau^2+\tau^4+2$. Either $\mathfrak{p}^{13} \mid 4\tau^2+\tau^4+2$ or $\mathfrak{p}^{13} \mid 4\tau^3+4\tau^2+\tau^4+2$. Hence, adding a fourth power $(\tau^3)^4 = \tau^{12}$, we obtain a good sum of four fourth powers which vanishes modulo $\mathfrak{p}^{13}$.

\textbf{Case \ref{cef}.11.} $\mathfrak{p}^{11} \parallel \tau^4+2$. Similarly as before, we have $s \geq 2$. Additionally, the following is true $\tau^8+2\tau^4 \equiv 0$ and $2\tau^6 \equiv 4\tau^2$ modulo $\mathfrak{p}^{13}$. Consider now good sums of three fourth powers from Lemma \ref{2.4}. After simplifying, we see that there is only possible remainder which can vanish modulo $\p^{13}$-- $4\tau^3+\tau^4+2$.  Hence, $s=2$ if and only if the aforementioned element vanishes. 

If $s>2$ then $\mathfrak{p}^{12} \parallel 4\tau^3+\tau^4+2$. Adding a fourth power $(\tau^3)^4 = \tau^{12}$ yields $s=3$.

\textbf{Case \ref{cef}.12.} $\mathfrak{p}^{12} \parallel \tau^4+2$. We have $(\tau^3)^4+\tau^4+2 \cdot 1^4 \equiv 0 \pmod{\mathfrak{p}^{13}}$, so $s \leq 3$. By Lemma \ref{2.3} we have $s>2$.

\textbf{Case \ref{cef}.13.} $\mathfrak{p}^{13} \mid \tau^4+2$. We have $\tau^4+ 2\cdot 1^4 \equiv 0 \pmod{\mathfrak{p}^{13}}$, so $s \leq 2$. By lemma \ref{2.2} we have $s > 1$. This finishes the proof.

\section{Computation of $s_4(K_{\mathfrak{p}})$} \label{step2}  

Up to this point we where considering an arbitrary quartic number field or arbitrary number field with a prime $\p \mid (2)$ with $e=4, \ f=1$. The next step for getting the lower bounds for $s_4$ form the theorems stated to this point is to compute how the ideal $(2)$ factorizes in $\OO_K$. This can be done by considering the basis of $\OO_K$ and computing the factorization from it. For a general number field of degree 4, the basis is not always explicitly known, and for when it is known, the number of cases of how it can look like is not particularly small (see section \ref{gen} for more discussion). Thus in this section, we will only compute lower bounds for the fourth level of biquadratic number fields of the form  $K = \mathbb{Q}(\sqrt{m},\sqrt{n})$ for $m,n \in \mathbb{Z}$, for which we have only 4 cases on how the basis of $\OO_K$ can look like.

Let us note that there is another square root in this biquadratic field.  Let $d = \gcd(m,n)$ and $k = \frac{mn}{d^2}$, then $\sqrt{k} \in K$. From \cite[Exercise 2.42]{marcus2018} we know how does the integral basis of $\mathcal{O}_K$ look like in biquadratic case. Now we will describe the general construction which will allow to determine how does the ideal $(2)$ factorizes in $\mathcal{O}_K$. For $K$ as above, its ring of integers $\mathcal{O}_K$ is of the following form
$$\mathcal{O}_K = \mathbb{Z} \oplus a\mathbb{Z}\oplus b\mathbb{Z}\oplus c \mathbb{Z} $$
for some $a,b,c \in K$. If the elements $a,b,c$ are explicitly given, then we may represent in this form each of the element $a^2,b^2,c^2,ab,bc,ac$. To be precise, each of the above elements can be expressed as a linear combination of $a,b,c$ with integral coefficients. With this, we are able to explicitly represent $\mathcal{O}_K$ as a quotient ring $\mathbb{Z}[x,y,z]/I$, where $I$ is the ideal generated by 6 linear combinations, when we replace $a$ with $x$, $b$ with $y$ and $c$ with $z$. We have obtained isomorphic rings with the isomorphism being the above identification.

Let us now discuss how to obtain factorization of the ideal $(2)$ in $\OO_K$. Consider the quotient ring of $\mathbb{Z}[x,y,z]/I$ divided by $(2)$. By the isomorphisms theorems this ring is isomorphic to $\mathbb{F}_2 [x,y,z]/\bar{I}$, where $\mathbb{F}_2$ is the field with two elements and $\bar{I}$ is the ideal $I$ with coefficients reduced modulo 2, and additionally to $\OO_K/(2)$. This is a finite ring with 16 elements and explicit calculations are possible. Hence, if we have $I$, it is a finite task to determine how does the ideal $(2)$ factorizes. We will often denote by $I$ the ideal of coefficients already reduced modulo 2. Up to rearranging the numbers $n,m,k$ we get 4 cases with some subcases.

We will now prove the second main result of the paper:

\begin{thm}\label{main2} Let $m,n$ be two square free integers. Let $K = \mathbb{Q}(\sqrt{m},\sqrt{n})$, $\mathfrak{p} \mid (2)$, $d = \gcd(m,n)$, $k = \frac{mn}{d^2}$. Then $s = s_4(K_{\mathfrak{p}})$ is one of $\{1,2,3,4,6,15\}$, specifically $($up to rearrangement of $m,n,k)$
\begin{center}
    if $m \equiv 3$, $n\equiv k \equiv 2 \pmod{4}$ then $s = \begin{cases}
        \textbf{1} \text{ if } 32 \mid n+k \text{ and } \frac{nk}{4} \equiv 15 \pmod{16} \\
        \text{or } 16 \parallel n+k \text{ and } \frac{nk}{4} \equiv 7 \pmod{16}\\
        \textbf{2} \text{ if } 32 \mid n+k \text{ and } \frac{nk}{4} \equiv 7 \pmod{16} \\
        \text{or } 16 \parallel n+k \text{ and } \frac{nk}{4} \equiv 15 \pmod{16}\\
        \textbf{3} \text{ if } 8 \parallel n+k
    \end{cases}$
\end{center}
\begin{center}
    if $m \equiv 1$, $n \equiv k \equiv 3 \pmod{4}$ then $s = \begin{cases}
        \textbf{4} \text{ if } m \equiv 1 \pmod{8}\\
        \textbf{2} \text{ if } m \equiv 5 \pmod{8}
    \end{cases}$
\end{center}
\begin{center}
    if $m \equiv 1$, $n \equiv k \equiv 2 \pmod{4}$ then $s = \begin{cases}
        \textbf{6} \text{ if } m \equiv 1 \pmod{8}\\
        \textbf{2} \text{ if } m \equiv 5 \pmod{8}
    \end{cases}$
\end{center}
\begin{center}
    if $m \equiv n \equiv k \equiv 1 \pmod{4}$ then $s = \begin{cases}
        \textbf{15} \text{ if } m \equiv n \equiv k \equiv 1 \pmod{8}\\
        \textbf{2} \text{ otherwise.}
    \end{cases}$
\end{center}
\end{thm}
\begin{proof}
Again, proof will be a case by case analysis. Moreover some computations were conducted in Macaulay2, especially when we need to explicitly know the structure of finite rings and its prime ideals.

\textbf{Case 3.1.} $m \equiv 3$, $n \equiv k \equiv 2 \pmod{4}$ and the basis of $\mathcal{O}_K$ is $\{1, \sqrt{m}, \sqrt{n}, \frac{\sqrt{n}+\sqrt{k}}{2}\} = \{1,a,b,c\}$. \\
We start with representing the following elements as integral linear combinations of the basis:
\begin{center}
\begin{tabular}{@{}l@{}}
    $a^2 = m \equiv 1 \pmod{2}$\\
    $b^2 = n \equiv 0 \pmod{2}$\\
    $c^2 = \frac{n+k}{4} + \frac{n}{2d}\sqrt{m} \equiv a \pmod{2}$\\
    $ab = d\sqrt{k} = 2d \cdot \frac{\sqrt{n} + \sqrt{k}}{2} - d\sqrt{n} \equiv b \pmod{2}$\\
    $ac = \frac{d\sqrt{k}+ \frac{m}{d} \sqrt{n}}{2} = d \cdot \frac{\sqrt{n} + \sqrt{k}}{2} + \frac{\frac{m}{d} - d}{2} \sqrt{n} \equiv c + b \pmod{2}$\\
    $bc = \frac{n + \frac{n}{d} \sqrt{m}}{2} \equiv 1 + a \pmod{2}.$
\end{tabular}
\end{center}
In the above calculations, we have used the fact that $8 \mid n+k$. If this is not the case then $n \equiv k \pmod{8}$, but $n \equiv k \equiv \frac{mn}{d^2} \pmod{8} \implies d^2 \equiv m \equiv 3 \pmod{4} $ a contradiction. 

We have obtained the ideal $I = (x^2+1,y^2,z^2+x,xy+y,xz+z+y,yz+1+x)$ with already reduced coefficients modulo 2. constructed from the above remainders modulo 2, we have that $\mathbb{F}_2[x,y,z]/I \cong \mathcal{O}_K/(2)$. Using the calculator Macaulay2 we will find prime ideals $P_1, \dots , P_r$ of $\mathbb{F}_2[x,y,z]/I$ and check what is the smallest power $\alpha$, so that $(P_1 \cdots P_r)^{\alpha} =(0)$. This will imply that $I = (P_1 \cdots P_r)^{\alpha}$ and $(2) = (\mathfrak{p}_1 \cdots \mathfrak{p}_r)^{\alpha}$ (the ramification indices must be equal, because biqadratic extension is Galois). By the introduction, we know what are the possible factorizations.

We will now proceed with computations in Macaulay2. At the very end we present first case on how we did it. Hence, in order to shorten the paper we will avoid repeating the code in each case.

For the ideal $I$ we get that the ideal $J =  (x+1,y,z+1)$ is prime with $J^2 = (x+y+1,x+1)$, $J^3 = (x+y+1)$ and $J^4= 0$ all in $\mathbb{F}_2[x,y,z]/I$. Here, $x,y,z$ of course denote their corresponding images in the quotient ring. $J^4= 0$ implies $(2) = \mathfrak{p}^4$, hence $e=4,f=1$ in $K$. We have  $z+1 \in J$ and $z+1 \not\in J^2$ hence, as the generator of $\p$ we may take $ c+1 = \tau \in \mathfrak{p}\setminus\mathfrak{p}^2$. Now, using $\tau = c+1$, $c^2 \equiv a$ and $a^2 = m \equiv 3 \pmod{4}$ we have to check which case from Theorem \ref{main} we can apply. See that $\tau^4 = (c+1)^4 = c^4+4c^3+6c^2+4c+1 \equiv c^4+2c^2+1 \equiv (c^2+1)^2 \equiv  (a+1)^2 \equiv a^2+2a+1 \equiv 2a \pmod{4 = \mathfrak{p}^8}$. Thus $\tau^4+2 \equiv 2(a+1) \pmod{4 = \mathfrak{p}^8}$. Now $x+1 \in J^2$ and $x+1 \not\in J^3$ implies $\mathfrak{p}^2 \parallel a+1 \implies \mathfrak{p}^6 \parallel \tau^4+2$.

By Theorem \ref{main} the fourth level $s$ depends on remainders modulo $\mathfrak{p}^{13}$ of the following:
\begin{center}
\begin{tabular}{@{}l@{}}
    $A_1 = 4\tau^3+4\tau^2+4\tau+(\tau^4+2\tau^2+2),$\\
    $A_2 = 4\tau + (\tau^4+1)(\tau^4+2\tau^2+2),$\\
    $A_3 = 4\tau + (\tau^4+2\tau^2+2), $\\
    $A_4= 4\tau^3+4\tau^2+4\tau+(\tau^4+1)(\tau^4+2\tau^2+2). $
\end{tabular}
\end{center}
Computation of $A_2,A_3,A_4$ will be based on the value of $A_1$.
\[\tau^4+2\tau^2+2 = (c+1)^4 + 2(c+1)^2+2 = c^4 + 4c^3 + 8c^2+8c+5 \equiv c^4+ 4c^3+8a+8c+5 \pmod{\mathfrak{p}^{13}},\]
\[4\tau^3 + 4\tau^2 + 4\tau = 4c^3+16c^2+24c+12 \equiv 4c^3+8c+12 \pmod{\mathfrak{p}^{13}} \]
\[ \implies A_1 = 4\tau^3+4\tau^2+4\tau+(\tau^4+2\tau^2+2) \equiv c^4+8c^3+8a+1 \equiv   \frac{(n+k)^2}{16} + \frac{nk}{4} + \frac{\frac{n}{d}(n+k)a}{4} + 8(c+b) + 8a + 1 \]
\[ \equiv \frac{(n+k)^2}{16} + \frac{nk}{4} + \frac{\frac{n}{d}(n+k)a}{4} + 1  \pmod{\mathfrak{p}^{13}}. \]
In the last congruence we used the fact that $8(a+b+c) \equiv 0 \pmod{\mathfrak{p}^{13}}$, and this follows from $(a+b+c) \equiv (a+b+c) + (a+b+1) \equiv c+1 \equiv 0 \pmod{\mathfrak{p}}$. We now have 2 cases:
\\ \\
$1^{\circ}$ $16 \mid n+k$. In this case $\frac{nk}{4} \equiv 7 \pmod{8}$ and $A_1 \equiv \frac{nk}{4} + \frac{\frac{n}{d}(n+k)a}{4} + 1  \pmod{\mathfrak{p}^{13}}$, hence $A_1 \equiv 7 + \frac{32a}{4}+ 1 \equiv 0 \pmod{8 = \mathfrak{p}^{12}}$. This readily implies $\mathfrak{p}^{9} \parallel \tau^4+2\tau^2+2$. This on the other hand implies $A_1 \equiv A_4, \  A_2 \equiv A_3 \pmod{\mathfrak{p}^{13}}$. We have $A_2 \equiv A_1 + 4\tau^2+4\tau^3 \pmod{\mathfrak{p}^{13}}$ which implies that $\mathfrak{p}^{10} \parallel A_2$. This gives that the value of $s$ depends only on $A_1 \pmod{\mathfrak{p}^{13}}$ - if $A_1 \equiv 0$ then $s=1$ and if not, then $\mathfrak{p}^{12} \parallel A_1$, so $s=2$ in this case.

If $32 \mid n+k$ and $\frac{nk}{4} \equiv 15 \pmod{16}$ or $16 \parallel n+k$ and $\frac{nk}{4} \equiv 7 \pmod{16}$ we have $A_1 \equiv 15 + 1 \equiv 0 \pmod{\p^{13}}$ or $A_1 \equiv 7 + 8a + 1 \equiv 8(a+1) \pmod{\p^{13}}$, respectively. In both cases we get $s=1$. 

If $32 \mid n+k$ and $\frac{nk}{4} \equiv 7 \pmod{16}$ or $16 \parallel n+k$ and $\frac{nk}{4} \equiv 15 \pmod{16}$ we have $\mathfrak{p}^{12} \parallel A_1$, hence $s=2$.
\\ \\
$2^{\circ}$ $8 \parallel n+k$. In this situation we have $\frac{nk}{4} \equiv 3 \pmod{8}$. Also, $A_1 \equiv 4 + 3 + 4a + 1 \equiv 4a \pmod{8 = \mathfrak{p}^{12}}$. Since $a \not\in \mathfrak{p}$ we have $\mathfrak{p}^8 \parallel A_1$. This implies that $\mathfrak{p}^8 \parallel A_i$ for $i=2,3,4$. By Theorem \ref{main}, to determine $s$ we need to check if any of $\{A_1+\tau^8, A_2+\tau^8, A_3+\tau^8, A_4+\tau^8\}$ is congruent to $0 \pmod{\mathfrak{p}^{12} = 8}$. We have the following chain of congruences:
\[ \tau^8 = (c+1)^8 = (c^2+2c+1)^4 = (c^2+1  + 2\gamma)^4  \equiv (a+1)^4 \]
\[  \equiv a^4 + 4a^3+6a^2+4a+1 \equiv m^2 + 4a(m+1) + 6m + 1 \equiv 1 + 18 + 1 \equiv 4 \pmod{8}. \]
Furthermore, since $a+1, \  a+b+1 \not\in (2)$ we obtain
\[ A_1 + \tau^8 \equiv 4a + 4 \equiv 4(a+1) \not\equiv 0 \pmod{8 },\]
\[A_3 + \tau^8 \equiv A_1 + \tau^8 + 4\tau^2+4\tau^3 \equiv 4a + 4 + 4(c+1)^2 + 4(c+1)^3 \equiv 4a + 4c^3 + 16c^2 + 20c + 12\]  \[\equiv 4a + 4(c+b) + 4c + 4 \equiv 4(a+b+1) \not\equiv 0 \pmod{8 }.\]
Also $A_1 \equiv A_4$ and $A_2 \equiv A_3 \pmod{8 = \mathfrak{p}^{12} }$, because $\mathfrak{p}^8 \parallel \tau^4+2\tau^2+2$ . Hence $s=3$ by Theorem \ref{main}.
\\ \\
\textbf{Case 3.2.} $m \equiv 1$, $n \equiv k \equiv 3 \pmod{4}$ and the basis of $\mathcal{O}_K$ is $\{1, \frac{1+\sqrt{m}}{2}, \sqrt{n}, \frac{\sqrt{n}+\sqrt{k}}{2}\} = \{1,a,b,c\}$. As before we calculate:
\begin{center}
\begin{tabular}{@{}l@{}}
$a^2 = \frac{m-1}{4} + \frac{1+\sqrt{m}}{2} \equiv a \text{ or } 1 + a \pmod{2}$ \\
$b^2 = n \equiv 1 \pmod{2}$\\
$c^2 = \frac{n+k+2\sqrt{nk}}{4} = \frac{n+k}{4} + \frac{n}{2d}\sqrt{m} = \frac{n+k-2\frac{n}{d}}{4} + \frac{n}{d} \cdot \frac{1+\sqrt{m}}{2} \equiv a \text{ or } 1+a \pmod{2}$\\
$ab = \frac{\sqrt{n} + d\sqrt{k}}{2} = d \cdot \frac{\sqrt{n}+\sqrt{k}}{2} - \frac{d-1}{2} \sqrt{n} \equiv c \text{ or } c + b \pmod{2} $\\
$ac = \frac{\sqrt{n} + \sqrt{k} + d\sqrt{k} + \frac{m}{d} \sqrt{n} }{4} =  \frac{d+1}{2} \cdot \frac{\sqrt{n}+\sqrt{k}}{2} + (\frac{\frac{m}{d}-d}{4})\sqrt{n} \equiv b+c \text{ or } c \text{ or } b \text{ or } 0 \pmod{2}$\\
$bc = \frac{n + \frac{n}{d}\sqrt{m}}{2} = \frac{n - \frac{n}{d}}{2} + \frac{n}{d} \cdot \frac{1+ \sqrt{m}}{2} \equiv a \text{ or }  1 + a \pmod{2}.$
\end{tabular}
\end{center}
The 'or' in $a^2$ depends on $m \pmod{8}$. The 'or' in $c^2$ depends on $n+k-2\frac{n}{d} \pmod{8}$. The 'or'-s in $ac$ depend on $m \pmod{8}$ and $d \pmod{4}$ and the remaining cases depend on $d \pmod{4}$. Note that $k \equiv mn \pmod{8}$, so by $d \pmod{4}$ and $m \pmod{8}$ we can compute $n+k-2\frac{n}{d} \equiv n+k-2\frac{3}{d} \pmod{8}$. If $m\equiv 1\pmod{8}$, then $n \equiv k \equiv 3 \pmod{4}$ and $k \equiv mn \equiv n \pmod{8}$ gives $n+k \equiv 6 \pmod{8}$. If $m \equiv 5 \pmod{8}$ then $n+k \equiv 2 \pmod{8}$. Hence we have to check 4 cases.

(a) $m \equiv 1 \pmod{8}$, $d \equiv 1 \pmod{4}$ $ \implies n+k-2\frac{n}{d} \equiv 0 \pmod{8}$. Similarly as before, we obtain an ideal $I = (x^2+x,y^2+1,z^2+x,xy+z,xz+z,yz+x)$. Again, using Macaulay2 we get the following data: ideals $A = (z,y+1,x)$ and $B = (z+1,y+1,x+1)$ are prime and $A^2 = (x,x+z,z)$, $B^2 = (x+1,x+y+z+1)$, $AB = (x+z,x+y+z+1)$ and $(AB)^2 = 0$. This all implies that $(2)=(\mathfrak{p}_1\mathfrak{p}_2)^2$ hence $ e=2, \ f=1$. Also, note that $y+1 \in A$ and $y+1 \not\in A^2$, hence $b+1 = \pi \in \mathfrak{p}_1 \setminus \mathfrak{p}_1^2$ and by Theorem \ref{1.3} $s = s_4(K_{\mathfrak{p}_1})$ depends on $\pi^2 = (b+1)^2 \pmod{\mathfrak{p}_1^4}$. We have $n \equiv 3 \pmod{\mathfrak{p}_1^4}$ and $y \not\in A \implies b = \sqrt{n} \not\in \mathfrak{p}_1$. 

Moreover, we have a chain of equalities and congruences $(b+1)^2 = (\sqrt{n}+1)^2 = n+1+2\sqrt{n} \equiv 2\sqrt{n} \not\equiv 2 \pmod{\mathfrak{p}_1^4}$. To justify this, see that there is an equivalence $2\sqrt{n} = 2b \equiv 2 \iff 2(b+1) \equiv 0 \pmod{\mathfrak{p}_1^4}$ which is does not hold since $\mathfrak{p}_1 \parallel b+1$, so by Theorem \ref{1.3} we have $s=4$ for $K_{\mathfrak{p}_1}$. For $\mathfrak{p}_2$ the situation is the same since the extension $K/\mathbb{Q}$ is Galois.
\\ \\
(b) $m \equiv 1 \pmod{8}$, $d \equiv 3 \pmod{4}$ $ \implies n+k-2\frac{n}{d} \equiv 4 \pmod{8}$. In this case we start with
$I = (x^2+x,y^2+1,z^2+x+1,xy+z+y,xz,yz+x+1)$ and we obtain prime ideals $A=(z,y+1,x+1)$ and $B = (z+1,y+1,x)$ satisfying $A^2 = (x+1,x+z+1,z)$, $B^2=(x,x+y+z)$, $AB = (x+z+1,x+y+z)$ and $(AB)^2=0$. As before we have $(2) = (\mathfrak{p}_1\mathfrak{p}_2)^2 \implies e=2, \ f=1$ and $b+1 = \pi \in \mathfrak{p}_1\setminus\mathfrak{p}_1^2$ same reasoning as in (a) yields $s=4$.
\\ \\
(c) $m \equiv 5 \pmod{8}$, $d \equiv 1 \pmod{4} \implies n+k-2\frac{n}{d} \equiv 4 \pmod{8}$. This time we have an ideal
$I = (x^2+x+1,y^2+1,z^2+x+1,xy+z,xz+z+y,yz+x)$ and a a single prime ideal $J = (y+1,x+z)$ satisfying $J^2 = 0.$ Hence, $ (2) = \mathfrak{p}^2 \implies e=2, \ f=2 \implies s=2$ by Theorem \ref{1.2}.
\\ \\
(d) $m \equiv 5 \pmod{8}$, $d \equiv 3 \pmod{4}$ $ \implies n+k-2\frac{n}{d} \equiv 0 \pmod{8}$. We get
$I = (x^2+x+1,y^2+1,z^2+x,xy+z+y,xz+y,yz+x+1)$ and a prime ideal $J = (y+1,x+z+1)$ satisfying $J^2 = 0.$ Thus $(2) = \mathfrak{p}^2 \implies e=2,\ f=2 \implies s=2$ by Theorem \ref{1.2}.
\\ \\
\textbf{Case 3.3.}  $m \equiv 1$, $n \equiv k \equiv 2 \pmod{4}$ and the basis of $\mathcal{O}_K$ is given by $\{1, \frac{1+\sqrt{m}}{2}, \sqrt{n}, \frac{\sqrt{n}+\sqrt{k}}{2}\} = \{1,a,b,c\}$.  As before, we calculate:
\begin{center}
\begin{tabular}{@{}l@{}}
$a^2 = \frac{m-1}{4} + \frac{1+\sqrt{m}}{2} \equiv a \text{ or } 1 + a\pmod{2}$, \\
$b^2 = n \equiv 0 \pmod{2}$,\\
$c^2 = \frac{n+k+2\sqrt{nk}}{4} = \frac{n+k}{4} + \frac{n}{2d}\sqrt{m} = \frac{n+k-2\frac{n}{d}}{4} + \frac{n}{d} \cdot \frac{1+\sqrt{m}}{2} \equiv 0 \pmod{2}$,\\

$ab = \frac{\sqrt{n} + d\sqrt{k}}{2} = d \cdot \frac{\sqrt{n}+\sqrt{k}}{2} - \frac{d-1}{2} \sqrt{n} \equiv c \text{ or } c+b \pmod{2}$,\\
$ac = \frac{\sqrt{n} + \sqrt{k} + d\sqrt{k} + \frac{m}{d} \sqrt{n} }{4} =  \frac{d+1}{2} \cdot \frac{\sqrt{n}+\sqrt{k}}{2} + (\frac{\frac{m}{d}-d}{4})\sqrt{n} \equiv 0 \text{ or } c \text{ or } b \text{ or } c+b \pmod{2}$,\\
$bc = \frac{n + \frac{n}{d}\sqrt{m}}{2} = \frac{n - \frac{n}{d}}{2} + \frac{n}{d} \cdot \frac{1+ \sqrt{m}}{2} \equiv 0 \pmod{2} $.
\end{tabular}
\end{center}
and all the 'or'-s in the above remainders depend only on $m \pmod{8}$ and $d \pmod{4}$, hence we have to check 4 cases.
\\ \\
(a) $m \equiv 1 \pmod{8}$, $d \equiv 1 \pmod{4}$. We construct an ideal
$I = (x^2+x,y^2,z^2,xy+z,xz+z,yz)$ and obtain prime ideals $A = (z,y,x)$ and $B = (z,y,x+1)$ satisfying $A^2 = (z,x)$, $B^2=(y+z,x+1)$, $AB = (y+z,z)$ and $(AB)^2 = 0$. Hence we have $(2) = (\mathfrak{p}_1\mathfrak{p}_2)^2 \implies e=2, \ f=1$.  Because $y \in A, \not\in A^2$ (the same is true for $\mathfrak{p}_2$) we can take $b = \pi \in \mathfrak{p}_1 \setminus \mathfrak{p}_1^2$. Then  $\pi^2 = b^2 = n \equiv 2 \pmod{\mathfrak{p}_1^4}$, hence by Theorem \ref{1.3} we have $s=6$, for $\mathfrak{p}_1$ and $\mathfrak{p}_2$.
\\ \\
(b) $m \equiv 1 \pmod{8}$, $d \equiv 3 \pmod{4}$. We have $I = (x^2+x,y^2,z^2,xy+z+y,xz,yz)$ and two prime ideals $A = (z,y,x)$ and $B=(z,y,x+1)$ with $A^2 = (y+z,x)$, $B^2 = (z,x+1)$, $AB = (y+z,z)$ and $(AB)^2 = 0$. As previously, we obtain $(2) = (\mathfrak{p}_1\mathfrak{p}_2)^2$ thus $e=2, \ f=1$. Take $b = \pi \in \mathfrak{p}_1 - \mathfrak{p}_1^2$ (which is also true for $\mathfrak{p}_2$). Same reasoning as in (a) shows that $s=6$ for both prime ideals.
\\ \\
(c) $m \equiv 5 \pmod{8}$, $d \equiv 1 \pmod{4}$. We have $I = (x^2+x+1,y^2,z^2,xy+z,xz+z+y,yz)$ and a prime ideal $J = (z,y)$ satisfying $J^2=0$. This gives $(2) = \mathfrak{p}^2 \implies e=2,\ f=2.$ Hence, $ s=2$ by Theorem \ref{1.2}.
\\ \\
(d) $m \equiv 5 \pmod{8}$, $d \equiv 3 \pmod{4}$. We have $I = (x^2+x+1,y^2,z^2,xy+z+y,xz+y,yz)$ and a prime ideal $J = (z,y)$ satisfying $J^2=0$. This gives $(2) = \mathfrak{p}^2 \implies e=2,\ f=2$. Hence, $s=2$ by theorem \ref{1.2}.
\\ \\
\textbf{Case 3.4.}  $ m\equiv n \equiv k \equiv 1 \pmod{4}$ and the basis of $\mathcal{O}_K$ is $\{1, \frac{1+\sqrt{m}}{2}, \frac{1+\sqrt{n}}{2},(\frac{1+\sqrt{m}}{2})(\frac{1+\sqrt{k}}{2})  \} = \{1,a,b,c\}$. 
\begin{center}
\begin{tabular}{@{}l@{}}
$a^2 = \frac{m-1}{4} + \frac{1+\sqrt{m}}{2}   $,\\
$b^2 = \frac{n-1}{4} + \frac{1+\sqrt{n}}{2}   $,\\
$c^2  =  \frac{(m-1)(k-1)}{16}  + \frac{m-1}{4} \cdot \frac{m+d}{2d} + \frac{k-m}{4}a + \frac{m}{d} \cdot \frac{1-m}{4}b + \frac{m+1}{2}c  $,\\
$ab = \frac{m-1}{4} + \frac{1-d}{2}a + \frac{1-m}{2}b + dc $,\\
$ac = \frac{m-1}{4} \cdot \frac{m+d}{2d} + \frac{1-m}{4}a + \frac{m}{d} \cdot \frac{1-m}{4}b + \frac{m+1}{2}c  $,\\
$bc = \frac{m-1}{4} \cdot \frac{n+d}{2d} +  \frac{\frac{n}{d}-d}{4}a + \frac{1-m}{4}b+ \frac{d+1}{2}c  $.
\end{tabular}
\end{center}
Because $k = \frac{mn}{d^2} \equiv mn \pmod{8}$ we have to distinguish four cases:
\\ \\
(a) $m \equiv n \equiv k \equiv 1 \pmod{8}$. We have $I = (x^2+x,y^2+y,z^2+z,xy+z,xz+z,yz+z)$. Using Macaulay2 one can find four different prime ideals $A = (z,y,x)$, $B = (z,y,x+1) $, $C = (z,y+1,x)$, $D = (z+1,y+1,x+1) $ such that $ABCD = 0.$ This implies that $ (2) = \mathfrak{p}_1\mathfrak{p}_2\mathfrak{p}_3\mathfrak{p}_4$. Hence for each $\mathfrak{p}_i$, $i=1,2,3,4$ we get $ e=f=1$ and $ s=15$ by Theorem \ref{1.1}.
\\ \\
(b) $m \equiv n \equiv 5,\  k \equiv 1 \pmod{8}$. Modulo 2 we get:
\begin{center}
\begin{tabular}{@{}l@{}}
$a^2 \equiv 1 +a$,\\
$b^2 \equiv 1 + b$,\\
$c^2 \equiv 1 +a + b + c \text{ or }  a + b + c$, \\
$ab \equiv 1+c \text{ or } 1+a+c$, \\
$ac \equiv  1+a + b + c \text{ or }  a + b + c$, \\
$bc \equiv 1+a + b + c \text{ or } a + b $.
\end{tabular}
\end{center}
Where all the 'or'-s depend on $d \pmod{4}$.

If $d \equiv 1 \pmod{4}$ then $I = (x^2+1+x,y^2+1+y,z^2+1+x+y+z,xy+1+z,xz+1+x+y+z,yz+1+x+y+z)$ and using Macaulay2 we get prime ideals $A=(z+1,y)$ and $B = (y+1,x+z+1)$ satisfying $AB = 0$. Hence $  (2) = \mathfrak{p}_1\mathfrak{p}_2 \implies e=1,f=2$ and $s=2$ by Theorem \ref{1.2}.

If $d \equiv 3 \pmod{4}$ then $I = (x^2+1+x,y^2+1+y,z^2+x+y+z,xy+1+x+z,xz+x+y+z,yz+x+y) $ and we get prime ideals $A = (z,x+y)$ and $B = (y+z+1,x+z)$ such that $AB = 0$. This implies $ (2) = \mathfrak{p}_1\mathfrak{p}_2 $ hence, $ e=1,f=2$ and $ s=2$ by Theorem \ref{1.2}.
\\ \\
The last remaining cases: (c) $m \equiv k \equiv 5, n \equiv 1 \pmod{8}$ and (d) $n \equiv k \equiv 5, m \equiv 1 \pmod{8}$ follow directly from rearranging $n,m,k$. In both cases we have $s=2$.

\end{proof}

\begin{remark}
In order to obtain the factorization of a given prime ideal $(p)$ one usually starts with the Kummer-Dedekind Theorem. Here, we have to find a different way since in almost all cases of biquadratic fields such an extenstion will not be monogenic, hence the aforementioned theorem does not apply.
\end{remark}

\section{Final remarks} \label{gen}

In this article we only computed explicit lower bounds for the fourth level of a very specific case of a quartic number field. In general for any quartic number field we have $ef \leq 4$ (for ramification index and inertial degree of any prime ideal of the integer ring), hence from Theorem \ref{main} combined with Theorems \ref{1.1}, \ref{1.2}, \ref{1.3}, \ref{1.4}, the fourth level of the $\p$-adic completion is computed for every possible case of $e$ and $f$. Thus if for a quartic number field the integral basis is known, using the method from section \ref{step2},  the computation of analogous lower bounds are possible. 

Integral bases for Galois quartic number fields are known \cite{albert1930}, but there are mistakes in the above paper, hence it is not reliable (see \cite{hw1990}).

Let us mention the case $s_4(K)=3$. Using Theorem \ref{main2} we get $s_4(\mathbb{Q}(\sqrt{-2},\sqrt{-6})) \geq 3$, because $8 \parallel -2 + (-6)$. As it turns out, $s_4(\mathbb{Q}(\sqrt{-2},\sqrt{-6}))$ is equal to 3. This is shown by the explicit formula $ \left(\frac{\sqrt{-2} + \sqrt{-6}}{2} \right)^4 + \left(\frac{\sqrt{-2} - \sqrt{-6}}{2} \right)^4 + \left(\sqrt{-2}+1\right)^4 + \left(\sqrt{-2}-1 \right)^4 = 0$. It is an open problem, which values can fourth level attain for fields (cf. \cite{hoffmann2014}). $\mathbb{Q}(\sqrt{-2},\sqrt{-6})$ seems to be the first known number field with fourth level equal to 3. Hence, we dare to formulate the following problem
\begin{prob}
Compute the possible values of $s_4(K)$ where $K$ is a number field.

\end{prob}
It is known that, if $s_4(K)$ is finite then $s_4(K) \leq 16$ by the result of Birch (see \cite{paper-per-piotr}). So far, it is still not known whether the fourth level can be equal to $5,7,8,9,10,11,12,13,14,16$ for number fields.

In every case when we are able to compute lower bounds for the fourth level and the fourth level itself, we always get an equality. As a consequence, we propose the following conjecture:

\begin{conj}
Let $K$ be a number field with a finite fourth level. Then
$$s_4(K)= \max_{\mathfrak{p} \mid (2)} s_4(K_\mathfrak{p}).$$
\end{conj}
Having a result like this, the computation of the fourth level of a number field could be reduced to a finite task, provided that we know how does the ideal $(2)$ factorizes.

\section{Acknowledgments}
At the time of writing this paper, the author is a high school student. The author would like to thank his tutor, Tomasz Kowalczyk, for suggesting this topic. Joint work of the author and Tomasz Kowalczyk is a part of the individual care of highly mathematically talented highschool students held by Jagiellonian University. Additionally, the author would like to thank Jakub Byszewski for suggesting the construction used in the Section 3.

\bibliographystyle{plain}
\bibliography{refs.bib}

\vspace{5pt}

%\noindent
%\textbf{Statements and Declarations}
%\\ \\
%\textbf{Funding} The author declares that no funds, grants, or other support were received %during the preparation of this manuscript. 
%\\ \\
%\textbf{Competing Interests} The author has no relevant financial or non-financial interests to disclose.
%\\ \\
%\textbf{Data Availability Statement} The author declares that the data supporting the findings of this study are available within the paper. The parts calculated using the algebraic calculator Macaulay2 can be replicated from the data provided. At the end of this article an example code from Macaulay2 is provided.

%\vspace{10pt}

\begin{small}

\noindent
Kazimierz Chomicz

\noindent
Institute of Mathematics

\noindent
Faculty of Mathematics and Computer Science

\noindent
Jagiellonian University

\noindent
ul. Łojasiewicza 6, 30-348 Kraków, Poland

\noindent
e-mail: kazikchomicz@icloud.com

\end{small}

\begin{verbatim}
Macaulay2, version 1.24.05-2230-g629880defb-dirty (vanilla)
with packages: ConwayPolynomials, Elimination, IntegralClosure, InverseSystems, 
Isomorphism, LLLBases, MinimalPrimes, OnlineLookup, Polyhedra, PrimaryDecomposition, 
ReesAlgebra, Saturation, TangentCone, Truncations, Varieties
i1 : k = ZZ/2[x,y,z]
o1 = k 
o1 : PolynomialRing 
i2 : I = ideal(x^2+1,y^2,z^2+x,x*y+y,x*z+z+y,y*z+1+x)
o2 = ideal(x^2+1,y^2,z^2+x,xy+y,xz+y+z,yz+x+1) 
o2 : Ideal of k 
i3 : S = k/I
o3 = S 
o3 : QuotientRing 
i4 : J =  ideal(x+1,y,z+1)
o4 = ideal(x+1,y,z+1) 
o4 : Ideal of S 
i5 : isPrime J
o5 = true 
i6 : J^2
o6 = ideal(x+y+1,x+1) 
o6 : Ideal of S 
i7 : J^3
o7 = ideal(x+y+1) 
o7 : Ideal of S 
i8 : J^4
o8 = ideal(0) 
o8 : Ideal of S
\end{verbatim}

\end{document}